\documentclass[a4paper,11pt]{article}
\title{Congruence relations for $p$-adic hypergeometric functions $\widehat{\mathscr{F}}_{a,...,a}^{(\sigma)}(t)$ and its transformation formula} 
\author{Wang Chung-Hsuan \\ Department of Mathematics, National Cheng Kung University \\ email: a78sddrt@gmail.com}

\usepackage[utf8]{inputenc}
\usepackage[english]{babel}
\usepackage{amssymb}
\usepackage{amsthm}
\usepackage{mathrsfs}
\usepackage{amsmath}  
\usepackage{graphics}  
\usepackage{centernot}
\usepackage{mathtools}
\usepackage{amsmath,amscd}
\usepackage[dvipsnames]{xcolor}
\usepackage[toc,page]{appendix}
\usepackage{tikz-cd}
\newtheorem{defi}{Definition}[section]
\newtheorem{thm}[defi]{Theorem}
\newtheorem{lem}[defi]{Lemma}
\newtheorem{cor}[defi]{Corollary}
\newtheorem{conj}[defi]{Conjecture}
\newtheorem{prop}[defi]{Proposition}
\newtheorem{remark}[defi]{Remark}

\theoremstyle{definition}
\newtheorem{eg}[defi]{Example}

\numberwithin{equation}{section}
\tikzset{
    labl/.style={anchor=south, rotate=90, inner sep=.5mm}
}
\providecommand{\keywords}[1]
{
  \textbf{\textit{Keywords---}} #1
}
\begin{document}
\date{}
\maketitle 
\begin{abstract}
  We introduce new kind of $p$-adic hypergeometric functions. We show these functions satisfy congruence relations similar to Dwork's $p$-adic hypergeometric functions, so they are convergent functions. And we show that there is a transformation formula between our new $p$-adic hypergeometric functions and $p$-adic hypergeometric functions of logarithmic type defined in \cite{A} in a particular case. 
\end{abstract}

\keywords
{$p$-adic hypergeometric functions, $p$-adic hypergeometric functions of logarithmic type, congruence relations}\medskip

\section{Introduction}
The classical hypergeometric function of one-variable
is defined to be the power series
$$
{}_sF_{s-1}\left({a_1,\ldots,a_s\atop b_1,\ldots,b_{s-1}};t\right)=\sum\limits_{k=0}^{\infty}\frac{(a_1)_k\cdots(a_s)_k}
{(b_1)_k\cdots(b_{s-1})_k}\frac{t^k}{k!}
$$
where $(\alpha)_k$ denotes the Pochhammer symbol (cf. \cite{slater}).
Let $(a_1,...,a_s)\in \mathbb{Z}_p^s$ be a $s$-tuple of $p$-adic integers, and consider
the series
$$
F_{a_1,...,a_s}(t)={}_sF_{s-1}\left({a_1,\ldots,a_s\atop 1,\ldots,1};t\right)=\sum\limits_{k=0}^{\infty}\frac{(a_1)_k}{k!}\cdots \frac{(a_s)_k}{k!}t^k.
$$
This is a formal power series with
$\mathbb{Z}_p$-coefficients. 
In the paper \cite{Dw}, B. Dwork introduced his $p$-adic hypergeometric functions.
Let $a^\prime$ be the Dwork prime of $a$, which is defined to be $(a+l) / p$ where $l\in \{0,1,...,p-1\}$ is the unique integer such that $a+l \equiv 0\mod p $. 
Then Dwork proved his $p$-adic hypergeometric function 
$
\mathscr{F}^{\rm Dw}_{a_1,...,a_s}(t)
$ (Definition \ref{Dwork's $p$-adic hypergeometric functions})
satisfies the congruence relations
$$
\mathscr{F}^{\rm Dw}_{a_1,...,a_s}(t)\equiv \frac{F_{a_1,...,a_s}(t)_{<p^n}}{[F_{a_1^\prime,...,a_s^\prime}(t^p)]_{<p^n}} \mod p^n\mathbb{Z}_p[[t]]
$$
where for a power series $f(t)=\sum_{n=0}^{\infty}A_nt^n$, we denote $f(t)_{<m}:=\sum_{n<m}A_nt^n$ the truncated polynomial. 
As a consequence, his function is $p$-adically analytic in the sense of Krasner
(i.e. an element of Tate algebra, \cite[3.1]{FP}), and one can define the value at
$t=\alpha$ with $|\alpha|_p=1$. Dwork applied his function to the unit root formula
of elliptic curve of Legendre type (e.g. \cite[\S 7]{Put}). \medskip

Recently, M. Asakura introduced a new function which he calls
$p$-adic hypergeoemtric functions of logarithmic type.
Let $W=W(\overline{\mathbb{F}}_{p})$ denote the Witt ring, and $K$=Frac$W$ its fractional field (e.g. \cite[Ch II, \S 6]{Se}). Let $\sigma : W[[t]] \rightarrow W[[t]]$ be a $p$-th Frobenius given by $\sigma(t)=ct^p$ with $c\in 1+pW:$ 
$$
\biggl(\sum_{i}a_{i}t^i\biggr)^{\sigma}=\sum_{i}a_{i}^{F}c^it^{ip}
$$ where $F: W \rightarrow W$ is the Frobenius on the Witt ring.
Then we can define Asakura's function $\mathscr{F}_{a_1,...,a_s}^{(\sigma)}(t)$
(see Definition \ref{p-adic hypergeometric functions of logarithmic type} or \cite[\S 3.1]{A}). If we write $\mathscr{F}_{a_1,...,a_s}^{(\sigma)}(t)=G_{a_1,...,a_s}(t)/F_{a_1,...,a_s}(t)$, he showed his $p$-adic hypergeometric functions of logarithmic type satisfy congruence relations similar to Dwork's,
$$
\mathscr{F}_{a_1,...,a_s}^{(\sigma)}(t)\equiv \frac{G_{a_1,...,a_s}(t)_{<p^n}}{F_{a_1,...,a_s}(t)_{<p^n}} \mod p^nW[[t]].
$$
Asakura studied his functions from the viewpoint of $p$-adic regulators.
In particular, some values of $\mathscr{F}_{a_1,...,a_s}^{(\sigma)}(t)$
are expected to be the special values of $p$-adic $L$-function of elliptic curves over $\mathbb Q$ (\cite[\S 5]{A}).
 
\medskip

In this paper, we introduce another new $p$-adic hypergeometric functions which we denote by $\widehat{\mathscr{F}}_{a,...,a}^{\;(\sigma)}(t)$. Let $a_1=\cdots=a_s=a.$ Then our function is defined as follows (Definition \ref{def-p-adicHG}),
$$
\widehat{\mathscr{F}}_{a,...,a}^{\;(\sigma)}(t):=\frac{t^{-a}}{F_{a,...,a}(t)}\int_{0}^{t} (t^aF_{a,...,a}(t)-(-1)^{se}[t^{a^\prime}F_{a^\prime,...,a^\prime}(t)]^{\sigma})\frac{dt}{t}.
$$
Write $\widehat{\mathscr{F}}_{a,...,a}^{\;(\sigma)}(t)=\widehat{G}^{(\sigma)}_{a,...,a}(t)/F_{a,...,a}(t)$. Our first main result is the following congruence relations
(Theorem \ref{congruence-thm}),
$$
\widehat{\mathscr{F}}_{a,...,a}^{\;(\sigma)}(t) \equiv \frac{\widehat{G}^{(\sigma)}_{a,...,a}(t)_{<p^n}}{F_{a,...,a}(t)_{<p^n}} \mod p^nW[[t]]. 
$$
The second main result is the transformation formula (=Theorem \ref{transformation-thm}). Let $\sigma(t)=ct^p$ and $\widehat{\sigma}(t)=c^{-1}t^p$. We will prove that
$$
\mathscr{F}_{a,a}^{\;(\sigma)}(t)=-\widehat{\mathscr{F}}_{a, a}^{\;(\widehat{\sigma})}(t^{-1})
$$
in case that $a\in \frac{1}{N}{\mathbb{Z}}$, $0<a<1$ and $p>N$ which comes from geometric observation (see next paragraph). We then generally conjecture 
a formula
$$
\mathscr{F}_{a,\cdots,a}^{\;(\sigma)}(t)=-\widehat{\mathscr{F}}_{a,\cdots,a}^{\;(\widehat{\sigma})}(t^{-1})
$$
between $\mathscr{F}_{a,\cdots,a}^{\;(\sigma)}(t)$ and $\widehat{\mathscr{F}}_{a,\cdots,a}^{\;(\widehat{\sigma})}(t)$, which we refer to as the transformation formula
(Conjecture \ref{transformation-conj}) and prove our theorem as a special case.
\medskip

The strategy of the proof of transformation formula is as follows. Let $N,A$ are integer with $N\geq 2$, $p>N,$ $1\leq A < N$ and gcd$(A,N)=1$. We consider the fibration $f:Y\to \mathbb{P}^1$ over $K=$Frac($W$) whose general fiber $f^{-1}(t)$ is the projective nonsingular model of
$$
y^N=x^A(1-x)^A(1-(1-t)x)^{N-A},
$$
and put $X_0:=f^{-1}(S)$ where $S_K:={\rm Spec}K[t,(t-t^2)^{-1}].$  \medskip

{
Asakura showed that his $\mathscr{F}_{a_1,...,a_s}^{(\sigma)}(t)$ appears in the $p$-adic regulator of $\xi$.
The key step in our proof is to provide an alternative description of the $p$-adic regulator by our $\widehat{\mathscr{F}}_{a,\cdots,a}^{\;(\sigma)}(t)$.
A key ingredient is the different description of $X$ via $(z,w,s_0)=(1-x,t_0^{A-N}y,t_0^{-1})$ with
$$
X: y^N=x^A(1-x)^A(1-(1-t_0^N)x)^{N-A}
$$
and
$$
\widehat{X}: w^N=z^A(1-z)^A(1-(1-s_0^N)z)^{N-A}
$$ with $t=t_0^N.$}\medskip

We also conjecture a similar transformation formula for Dwork's hypergeometric functions (Conjecture \ref{Dwork-trans-conj}). We attach a proof in case $s=2$, $a\in \frac{1}{N}{\mathbb{Z}}$, $0<a<1$ and $p>N$ which is based on a similar idea to the above (Theorem \ref{Dwork-trans-thm}).\medskip

This paper is organized as follows. In \S 2, we give the definition of $\widehat{\mathscr{F}}_{a,...,a}^{\;(\sigma)}(t)$. In \S 3, we prove the congruence relations. In \S 4, we give the proof of the transformation formula.
We also give a brief review on hypergeometric curves and some results in \cite{A} which are necessary
in the proof.

\medskip

\textbf{Acknowledgement.} I truly appreciate the help of Professor Masanori Asakura. He gave the definition of our new $p$-adic hypergeometric functions and the conjectures of transformation formulas with the aid of computer. Also, he gave a lot of advice of this paper. 
\section{Definition of $\widehat{\mathscr{F}}_{a,...,a}^{\;(\sigma)}(t)$}
Let $W=W(\overline{\mathbb{F}}_{p})$ denote the Witt ring, and $K$=Frac$W$ its fractional field. Let $\sigma : W[[t]] \rightarrow W[[t]]$ be a $p$-th Frobenius given by $\sigma(t)=ct^p$ with $c\in 1+pW$ :

\begin{equation}
\biggl(\sum_{i}a_{i}t^i\biggr)^{\sigma}=\sum_{i}a_{i}^{F}c^it^{ip}
\end{equation}
where $F: W \rightarrow W$ is the Frobenius on $W$. Given $a \in \mathbb{Z}_p$. Let $a^\prime$ be the Dwork prime of $a$, which is defined to be $(a+l) / p$ where $l\in \{0,1,...,p-1\}$ is the unique integer such that $a+l \equiv 0\mod p $. We denote the $i$-th Dwork prime by $a^{(i)}$ which is defined to be $(a^{(i-1)})^\prime$ with $a^{(0)}=a.$\medskip

Let 
\begin{equation}
F_{a,...,a}(t)=\sum_{k=0}^{\infty}{\biggl( \frac{(a)_k}{k!} \biggr)}^s t^k , \quad F_{a^\prime,...,a^\prime}(t)=\sum_{k=0}^{\infty}{\biggl( \frac{(a^\prime)_k}{k!} \biggr)}^s t^k
\end{equation}
be the hypergeometric series for $a\in \mathbb{Z}_p$, where $(a)_k$ denotes the Pochhammer symbol (i.e. $(\alpha)_k=\alpha(\alpha +1)\cdots (\alpha +k+1)$ when $k\neq 0$ and $(\alpha)_0=1$).\medskip

Put
$$ q:=\left\{
\begin{aligned}
4 &\quad p=2 \\
p  &\quad p\geq3. \\
\end{aligned}
\right.
$$
Let $l^\prime\in \{0,1,...,q-1\}$ be the unique integer such that $a+l^\prime \equiv 0\mod q $. \medskip

Put
$$
e:=l^\prime-\lfloor \frac{l^\prime}{p} \rfloor.
$$
Define a power series 
$$
\begin{aligned}
\widehat{G}^{(\sigma)}_{a,...,a}(t)&:=t^{-a} \int_{0}^{t} (t^aF_{a,...,a}(t)-(-1)^{se}[t^{a^\prime}F_{a^\prime,...,a^\prime}(t)]^{\sigma})\frac{dt}{t} \\
&=\sum_{k=0}^{\infty}B_{k}t^k
\end{aligned} 
$$
for $a\in \mathbb{Z}_p \backslash \mathbb{Z}_{\leq 0}$. Here we think $t^{\alpha}$ to be an abstract symbol with relations $t^{\alpha}\cdot t^{\beta}= t^{\alpha+\beta}$, on which $\sigma$ acts by $\sigma(t^{\alpha})=c^{\alpha}t^{p\alpha}$, where 
$$
c^{\alpha}=(1+pu)^{\alpha}:=\sum_{i=0}^{\infty}\binom{\alpha}{i}p^{i}u^i, \quad u\in W.  
$$
Moreover we think $\int_{0}^{t}(-)\frac{dt}{t}$ to be a operator such that
$$
\int_{0}^{t}t^{\alpha}\frac{dt}{t}=\frac{t^{\alpha}}{\alpha}, \quad \alpha \neq 0. 
$$
\begin{defi}\label{def-p-adicHG} Define
$$\widehat{\mathscr{F}}_{a,...,a}^{\;(\sigma)}(t):=\frac{\widehat{G}^{(\sigma)}_{a,...,a}(t)}{F_{a,...,a}(t)}, \quad a\in \mathbb{Z}_p \backslash \mathbb{Z}_{\leq 0}.$$ 
\end{defi}\noindent If we write $F_{a,...,a}(t)=\sum A_{k}t^k$ and $F_{a^\prime,...,a^\prime}(t)=\sum A_k^{(1)}t^k$. Then we have 
$$
B_k=\frac{1}{k+a}\biggl(A_{k}-(-1)^{se}(A_{\frac{k-l}{p}}^{(1)})c^{\frac{k+a}{p}}\biggr),
$$
where $A_{\frac{m}{p}}^{(1)}=0$ if $m \not\equiv 0 \mod p$ or $m <0$.
In fact, we have $\widehat{\mathscr{F}}_{a,...,a}^{\;(\sigma)}(t)$ and $\widehat{G}^{(\sigma)}_{a,...,a}(t)$ are power series with $W$-coefficients from the following lemma.
\begin{lem}
We have
$$
B_k \in W, \quad \forall k \in \mathbb{Z}_{\geq 0}.
$$
\end{lem}

\begin{proof}
If $k\not\equiv l \mod p$, then $k+a \not\equiv 0 \mod p.$ Hence
$$
B_k=\frac{A_k}{k+a} \in W.
$$

For $k\equiv l \mod p$, write $k+a=p^nm$ with $p \nmid m.$ Then 
$$
c^{\frac{k+a}{p}}=c^{p^{n-1}m}\equiv 1 \mod p^n
$$
since $c \in 1+pW$. So if we can show 
$$
A_k\equiv (-1)^{se}A_{\frac{k-l}{p}}^{(1)} \mod p^n
$$
then $B_k\in W.$ Indeed, for an $p$-adic integer $\alpha\in \mathbb{Z}_p$ and $n\in \mathbb{Z}_{\geq 1}$ we define 
$$
\{\alpha\}_n:=\prod\limits_{\substack {1\leq i \leq n \\ p\nmid (\alpha+i-1)}}(\alpha+i-1)
$$
Using \cite[Lemma 3.6]{A}, one have

$$
\frac{(a)_k}{k!}=\bigg(\frac{(a^\prime)_{\lfloor k/p \rfloor}}{\lfloor k/p \rfloor}!\bigg)^{-1}\frac{\{a\}_k}{\{1\}_k}=\frac{(a^\prime)_{\frac{k-l}{p}}}{(\frac{k-l}{p})!}\frac{\{a\}_k}{\{1\}_k}.
$$
Therefore
$$
A_k=A_{\frac{k-l}{p}}^{(1)}\bigg(\frac{\{a\}_k}{\{1\}_k}\bigg)^{s}.
$$

Since $k+a\equiv 0 \mod p^n$, we have that $\{a\}_k\equiv \{-k\}_k \mod p^n.$ Therefore
$$
\begin{aligned}
\frac{\{a\}_k}{\{1\}_k}&\equiv \frac{\{-k\}_k}{\{1\}_k} \\
&\equiv \frac{\prod\limits_{\substack{ 1\leq i \leq k \\ p \nmid (-k+i-1)}}(-k+i-1)}{\prod\limits_{\substack {1\leq i \leq k \\ p\nmid i}}i}\\
&\equiv \prod\limits_{\substack{i=1 \\ p\nmid i}}^{k}(-1)=(-1)^{k-\lfloor k/p \rfloor} \mod p^n.
\end{aligned}
$$

We claim $(-1)^{k-\lfloor k/p \rfloor}=(-1)^e.$\medskip

$\bullet$ For $p \geq 3$, write $k=l+bp$. Then 
$$
k-\lfloor k/p \rfloor=l+bp-\lfloor l/p+b \rfloor=e+b(p-1)\equiv e \mod 2.
$$
Therefore  $(-1)^{k-\lfloor k/p \rfloor}=(-1)^e.$ \medskip

$\bullet$ For $p=2$, write $k+a=2^nm,$ $2 \nmid m.$\medskip

(1)If $n=1,$ we have $(-1)^{k-\lfloor k/p \rfloor}=(-1)^e.$ \medskip

(2)If $n\geq 2,$ we have $k+a\equiv 0 \mod 4.$ Therefore $k\equiv L \mod 4$ where $L\in\{0,1,2,3\}$ is the unique integer such that $a+L\equiv 0 \mod 4.$ Write $k=L+4b$, then 
$$
k-\lfloor k/2 \rfloor=L+4b-\lfloor L/2 \rfloor-2b\equiv e \mod 2.
$$
Again, we obtain $(-1)^{k-\lfloor k/p \rfloor}=(-1)^e$. \medskip

Therefore
$$
A_k\equiv A_{\frac{k-l}{p}}^{(1)}(-1)^{es} \mod p^n 
$$

This completes the proof.
\end{proof}
\section{Congruence Relations}
For a power series $f(t)=\sum_{n=0}^{\infty}A_nt^n$, we denote $f(t)_{<m}:=\sum_{n<m}A_nt^n$ the truncated polynomial. Then we have the following theorem which we call the congruence relations of $\widehat{\mathscr{F}}_{a,...,a}^{\;(\sigma)}(t).$
\begin{thm}\label{congruence-thm}
Let $a\in \mathbb{Z}_p \backslash \mathbb{Z}_{\leq 0}$ and suppose $c\in1+qW$. Then
$$
\widehat{\mathscr{F}}_{a,...,a}^{\;(\sigma)}(t) \equiv \frac{\widehat{G}^{(\sigma)}_{a,...,a}(t)_{<p^n}}{F_{a,...,a}(t)_{<p^n}} \mod p^nW[[t]] 
$$ for all $n \in \mathbb{Z}_{\geq0}$.
\end{thm}
\begin{cor}\label{overconvergent}
If $a^{(r)}=a$ for some $r>0,$ where $(-)^{(r)}$ denote the $r$-th Dwork prime, then
$$
\widehat{\mathscr{F}}_{a,...,a}^{\;(\sigma)}(t)\in W\langle t, t^{-1}, h(t)^{-1} \rangle, \quad h(t):=\prod\limits_{i=0}^{r-1}F_{a^{(i)},...,a^{(i)}}(t)_{<p}
$$
is a convergent function, where $W\langle t, t^{-1}, h(t)^{-1} \rangle :=\underset{n}{\varprojlim}(W/p^n[t,t^{-1},h(t)^{-1}])$.
\end{cor}
\begin{proof}
It suffices to show that modulo $p^n$
$$
\frac{1}{F_{a,\cdots,a}(t)_{<p^n}} 
$$ lies in $W/p^n[t,t^{-1},h(t)^{-1}]$. From Dwork congruence \cite[p.37, Theorem 2]{Dw}, we have
$$
F_{a,\cdots,a}(t)_{<p}\equiv \frac{F_{a,\cdots,a}(t)_{<p^n}}{\Big[F_{a^\prime,\cdots,a^\prime}(t^p)\Big]_{<p^n}} \mod p\mathbb{Z}_p[[t]].
$$ Thus,
$$
F_{a,\cdots,a}(t)_{<p^n}\equiv F_{a,\cdots,a}(t)_{<p}\Big[F_{a^\prime,\cdots,a^\prime}(t^p)\Big]_{<p^n} \mod p\mathbb{Z}_p[[t]].
$$ Using Dwork congruence again, we obtain
$$
F_{a,\cdots,a}(t)_{<p}\Big[F_{a^\prime,\cdots,a^\prime}(t^p)\Big]_{<p^n}\equiv F_{a,\cdots,a}(t)_{<p}\Big[F_{a^\prime,\cdots,a^\prime}(t^p)\Big]_{<p^2}\Big[F_{a^{(2)},\cdots,a^{(2)}}(t^{p^2})\Big]_{<p^n} \mod p\mathbb{Z}_p[[t]].
$$ So by applying Dwork congruence inductively, we have
$$
F_{a,\cdots,a}(t)_{<p^n}\equiv F_{a,\cdots,a}(t)_{<p}\cdots \Big[F_{a^{(i)},\cdots,a^{(i)}}(t^{p^i})\Big]_{<p^{i+1}}\cdots\Big[F_{a^{(n-1)},\cdots,a^{(n-1)}}(t^{p^{n-1}})\Big]_{<p^{n}} \mod p\mathbb{Z}_p[[t]].
$$ Since both sides of the equation are polynomial, $p\mathbb{Z}_p[[t]]$ can be replaced by $p\mathbb{Z}_p[t]$.
Now we use the fact that for any $F(x)\in \mathbb{Z}_p[[t]],$ one has
$$
\left.F(x)_{<p^{n-1}}\right|_{x=t^p}\equiv \bigg(\left.F(x)_{<p^{n-1}}\right|_{x=t}\bigg)^p= \bigg(F(t)_{<p^{n-1}}\bigg)^p \mod p\mathbb{Z}_p[t]. 
$$ Then by induction, we have
$$
\Big[F_{a^{(i)},\cdots,a^{(i)}}(t^{p^i})\Big]_{<p^{i+1}}\equiv \Big[F_{a^{(i)},\cdots,a^{(i)}}(t)_{<p}\Big]^{p^{i}} \mod p\mathbb{Z}_p[t]
$$ Therefore,
$$
F_{a,\cdots,a}(t)_{<p^n}\equiv F_{a,\cdots,a}(t)_{<p}\cdots \Big[F_{a^{(i)},\cdots,a^{(i)}}(t)_{<p}\Big]^{p^i}\cdots\Big[F_{a^{(n-1)},\cdots,a^{(n-1)}}(t)_{<p}\Big]^{p^{n-1}} \mod p\mathbb{Z}_p[t].
$$
So we can write $F_{a,\cdots,a}(t)_{<p^n}=F_{a,\cdots,a}(t)_{<p}\cdots \Big[F_{a^{(i)},\cdots,a^{(i)}}(t)_{<p}\Big]^{p^i}\cdots\Big[F_{a^{(n-1)},\cdots,a^{(n-1)}}(t)_{<p}\Big]^{p^{n-1}}+pg(t)$ for some $g(t)\in \mathbb{Z}_p[t].$ Then considering
$(F_{a,\cdots,a}(t)_{<p^n})^{-1}$ modulo $p^n$, we see that it lies in $W/p^n[t,t^{-1},h(t)^{-1}].$
\end{proof}
\subsection{Proof of Congruence Relations: Lipschitz Functions and Congruence Relations for ${B_k}/{A_k}$}
Here we introduce a notion called Lipschitz functions.
\begin{defi} Let $W$ be the Witt ring of $\overline{\mathbb{F}}_p.$ A function
A function $f:\mathbb{Z}_{\geq 0} \rightarrow \text{Frac}(W)$ is called \textit{Lipschitz} if it takes values in $W$ and $p^s|(n-m)$ implies $p^s|(f(n)-f(m))$ for any $n,m,s\in \mathbb{Z}_{\geq 0}.$ 
\end{defi}

\begin{remark}
Generally, one can define Lipschitz functions as functions from $X$ to $K$ where $K$ denote a complete extension of $\mathbb{Q}_p$ and $X\subseteq K$ is a subset with no isolated point and satisfy
$$
|f(x)-f(y)|_p\leq M|x-y|_p
$$ for some constant $M$ and $x,y\in X$(See \cite[\S 5.1]{Rob}). In our case, $M=1.$
\end{remark}
In this section, we fix $a\in \mathbb{Z}_p\backslash \mathbb{Z}_{\leq 0}$. And the main result of this section is the following lemma.
\begin{lem}\label{BA_congruence}
The function defined by $k\mapsto B_k/A_k$ is Lipschitz.
\end{lem}

Before proving Lemma \ref{BA_congruence}, we need some preliminary lemmas. 
\begin{lem}\label{Lsumproduct}
Polynomial functions are Lipschitz. The sums and products of Lipschitz functions are Lipschitz. 
\end{lem}

\begin{lem}\label{c=1}
Lemma \ref{BA_congruence} can be reduced to the case $\sigma(t)=t^p,$ i.e., $c=1$.  
\end{lem}
\begin{proof}
Suppose that the function defined by $n\mapsto B_n^{\circ}/A_n$ is Lipschitz. Let $B_n^{\circ}$ be the coefficients of $\widehat{G}_{a,\cdots,a}^{(\sigma)}(t)$ with respect to $c=1.$ Then suppose $n\mapsto B_{n}^\circ/A_n$ is Lipschitz. Then
$$
n\mapsto \frac{A^{(1)_{\frac{n-l}{p}}}}{A_n}=(-1)^{se+1}\bigg((n+a)\frac{B_n^\circ}{A_n}-1\bigg)
$$ is also Lipschitz by Lemma \ref{Lsumproduct}. Observe that 
$$
\frac{B_n}{A_n}=\frac{B_n^\circ}{A_n}+(-1)^{se+1}\bigg(\frac{c^{\frac{n+a}{p}}-1}{n+a}\bigg)\frac{A^{(1)}_{\frac{n-l}{p}}}{A_n}.
$$ So to prove $n\mapsto B_n/A_n$ is Lipschitz, it suffices to prove $n \mapsto \frac{c^{n/p}-1}{n}$ is Lipschitz(here the value is $0$ when $p\nmid n).$ Write $c=1+kq$ for some $k\in W.$ Then for those $n$ with $p| n$, we have 
$$
\frac{c^{n/p}-1}{n}=\frac{(1+kq)^{n/p}-1}{n}=\frac{1}{n}\sum_{i=1}^\infty \binom{n/p}{i}k^iq^i.
$$ So to prove $n \mapsto \frac{c^{n/p}-1}{n}$ is Lipschitz, it suffices to show that
$$
\frac{1}{n_1}\binom{n_1/p}{i}k^iq^i\equiv \frac{1}{n_2}\binom{n_2/p}{i}k^iq^i \mod p^m
$$ for $n_1\equiv n_2 \mod p^m.$ For $i=1$, they are equal. Suppose $i\geq 2,$ we have
$$
    \frac{1}{n_1}\binom{n_1/p}{i}k^i q^i=\bigg(\frac{n_1}{p}-1\bigg)\cdots\bigg(\frac{n_1}{p}-i+1\bigg)\frac{k^i q^i}{i!p}.
$$
By \cite[\S 5.1 p.49]{W}, we know 
$$
v_p(i!)<\frac{i}{p-1}.
$$ Thus
$$
v_p(\frac{k^i q^i}{i!p})\geq 1.
$$
So
$$
\begin{aligned}
\bigg(\frac{n_1}{p}-1\bigg)\cdots\bigg(\frac{n_1}{p}-i+1\bigg)\frac{k^i q^i}{i!p}&\equiv
\bigg(\frac{n_2}{p}-1\bigg)\cdots\bigg(\frac{n_2}{p}-i+1\bigg)\frac{k^iq^i}{i!p} \mod p^m \\
&=\frac{1}{n_2}\binom{n_2/p}{i}k^iq^i.
\end{aligned}
$$
Therefore, the result follows.
\end{proof}
\begin{lem}\label{s=1}
Assuming $\sigma(t)=t^p$, lemma \ref{BA_congruence} can be reduced to $s=1.$
\end{lem}
\begin{proof}
Assume that $s=1$ and $\sigma(t)=t^p$  and suppose we have proved that $n\mapsto B_n/A_n$ is Lipschitz.
Then $n\mapsto A^{(1)}_{\frac{n-l}{p}}/A_n$ is also Lipschitz (see the proof of Lemma \ref{c=1}). Now for arbitraty $s\in \mathbb{Z}_{>0}$, we have the function
$$
\begin{aligned}
n\mapsto &\frac{1}{n+a}\Bigg(\frac{A_n^s-((-1)^{e}A_{\frac{n-l}{p}}^{(1)})^s}{A_n^s}\Bigg)\\
&=\frac{1}{n+a}\Bigg(\frac{A_n-(-1)^{e}A_{\frac{n-l}{p}}^{(1)}}{A_n}\Bigg)\sum_{j=0}^{s-1}\Bigg(\frac{(-1)^e A^{(1)}_{\frac{n-l}{p}}}{A_n}\Bigg)^j \\
&=\frac{B_n}{A_n}\sum_{j=0}^{s-1}\Bigg(\frac{(-1)^e A^{(1)}_{\frac{n-l}{p}}}{A_n}\Bigg)^j
\end{aligned}
$$ is also Lipschitz by Lemma \ref{Lsumproduct}.
\end{proof}

By Lemma \ref{c=1} and Lemma \ref{s=1}, from now on, we assume that $\sigma(t)=t^p$ and $s=1.$

\begin{lem}\label{BlAl}
If $a+l=cp^n$ with $p \centernot \mid c$, then
$$
\frac{B_l}{A_l}\equiv {\psi}_p(a+l)-{\psi}_p(1+l)\equiv -(\psi_p(1+l)+\gamma_p) \mod p^n,
$$
\end{lem} 
where $\psi_p(z)$ is the $p$-adic digamma function and $\gamma_p$ is the $p$-adic Euler constant (\cite[\S 2.2]{A}).
\begin{proof} \boxed{\text{Case I : } p \neq 2} \medskip

First, if $l=0$, then $e=l=0.$ Therefore we have
$$
\frac{B_l}{A_l}=\frac{B_0}{A_0}=\frac{1}{a}(1-(-1)^{e})=0
$$
and 
$$
{\psi}_p(a+l)-{\psi}_p(1+l)\equiv {\psi}_p(0)-{\psi}_p(1)=0 \mod p^n.
$$
If $l \neq 0$, then
$$
\begin{aligned}
\frac{B_l}{A_l}&=\frac{1}{a+l}\Big(1-(-1)^{e}\frac{1}{A_l}\Big) \\
&=\frac{1}{cp^n}\bigg(1-(-1)^{e}\frac{l!}{(a)_l}\bigg).
\end{aligned}
$$
Arranging the formula above and using $a+l=cp^n,$ we obtain
$$
\begin{aligned}
(-1)^{e}\Big(1-cp^n\frac{B_l}{A_l}\Big)&=(-1)^l\frac{1\cdot2\cdots l}{(l-cp^n)\cdots(1-cp^n)} \\
&=(-1)^l\prod_{1\leq k \leq l}\bigg(\frac{1}{1-cp^n/k}\bigg) \\
&\equiv (-1)^l\bigg(1+\sum_{1\leq k \leq l}\frac{cp^n}{k}\bigg) \mod p^{2n} \\
&\equiv (-1)^{l}\bigg(1+cp^n\sum_{1\leq k \leq l}\frac{1}{k}\bigg) \mod p^{2n} \\
&=(-1)^{l}\bigg(1+cp^n(\psi_p(1+l)+\gamma_p)\bigg).
\end{aligned}
$$ 
Therefore, we have
$$
1-cp^n\frac{B_l}{A_l}\equiv (-1)^{(e-l)}\Big(1+cp^n(\psi_p(1+l)+\gamma_p)\Big) \mod p^{2n}.
$$
Since we assume $p$ is odd, $e$ is equal to $l$. Hence 
$$
\frac{B_l}{A_l}\equiv -(\psi_p(1+l)+\gamma_p)\equiv \psi_p(a+l)-\psi_p(1+l) \mod p^n.
$$
by using the formula in \cite[(2.13) and Theorem 2.4]{A}. \medskip

\boxed{\text{Case II : } p=2} \medskip

Again, we have
$$
\frac{B_l}{A_l}=\frac{1}{a+l}\Big(1-(-1)^{e}\frac{1}{A_l}\Big).
$$
(1) For $a=2^nc$ with $2\centernot \mid c$, we have $l=0.$ Hence
$$
\frac{B_l}{A_l}=\frac{B_0}{A_0}=\frac{1}{a}(1-(-1)^{e})=\frac{1}{2^nc}(1-(-1)^{e}).
$$
If $n\geq 2$, then $l^\prime=0$ and $e=0$. Therefore, we obtain
$$
\frac{B_0}{A_0}=0, \quad \psi_p(a)-\psi_p(1)\equiv 0 \mod p^n
$$
by \cite[(2.13) Case II]{A}. \medskip

If $n=1,$ then $l^\prime=2$ and $e=1$. This implies
$$
\frac{B_0}{A_0}=\frac{1}{2c}(1-(-1))\equiv 1 \mod 2
$$ and

$$
\psi_p(a)-\psi_p(1)= \psi_p(2c)-\psi_p(0)\equiv c \equiv 1 \mod 2.
$$ Hence, in both case, we have
$$
\frac{B_l}{A_l}\equiv {\psi}_p(a+l)-{\psi}_p(1+l) \mod p^n.
$$
(2) For $a+1=2^nc$ with $2 \centernot \mid c$, we have $l=1.$ This implies
$$
\frac{B_l}{A_l}=\frac{B_1}{A_1}=\frac{1}{a+1}(1-(-1)^{e}\frac{1}{A_1})=\frac{1}{2^nc}(1-(-1)^{e}\frac{1}{a}).
$$
If $n\geq 2,$ then $l^\prime=1$ and $e=1.$ We obtain
$$
\begin{aligned}
2^nc\frac{B_1}{A_1}&=(1+\frac{1}{a})\\
&=(1-\frac{1}{1-2^nc}) \\
&\equiv1-(1+2^nc)=-2^nc  \mod 2^{2n},
\end{aligned}
$$ which implies

$$
\frac{B_1}{A_1}\equiv -1 \mod 2^n.
$$

Also, we have
$$
\psi_p(a+1)-\psi_p(2)\equiv \psi_p(0)-\psi_p(2)\equiv -1 \mod 2^n.
$$
If $n=1$, then $l^\prime=3$ and $e=2.$ We have
$$
\begin{aligned}
2c\frac{B_1}{A_1}&=1+\bigg(\frac{1}{1-2c}\bigg)\\
&\equiv 1+(1+2c) \mod 4 \\
&\equiv 0 \mod 4.
\end{aligned}
$$
Hence, we get
$$
 \frac{B_1}{A_1}\equiv0 \mod 2.
$$
On the other hand, we have
$$
\begin{aligned}
\psi_p(a+1)-\psi_p(2)&\equiv 
\psi_p(2c)-\psi_p(2) \\
&\equiv c-1 \\
&\equiv 0 \mod 2.
\end{aligned}
$$
\end{proof}
\begin{lem}\label{B_k_and_B_l}
If $k=l+bp^m$ with $p \centernot \mid b,$ then
$$
\frac{B_k}{A_k}\equiv\frac{B_l}{A_l} \mod p^m.
$$
\end{lem}
\begin{proof}
Again, we write $a+l=cp^n$ with $p \centernot \mid c.$\medskip

\boxed{\text{Case I : } m\neq n} \medskip

We have that
$$
\begin{aligned}
1-(a+l+bp^m)\frac{B_{l+bp^m}}{A_{l+bp^m}}&=(-1)^{e}\frac{A_{bp^{m-1}}^{(1)}}{A_{l+bp^m}}\\
&\overset{\mathrm{(*)}}{=} (-1)^{e}\frac{\{1\}_{l+bp^m}}{\{a\}_{l+bp^m}} \\
&= (-1)^{e}\frac{\{1\}_{l}}{\{a\}_{l}}\cdot\frac{\{1+l\}_{bp^m}}{\{1\}_{bp^m}}\frac{\{1\}_{bp^m}}{\{a+l\}_{bp^m}} \\
&\overset{\mathrm{(**)}}{\equiv}  (-1)^{e}\frac{1}{A_{l}}\Big(1-bp^m(\psi_p(a+l)-\psi_p(1+l))\Big) \mod p^{2m} \\
&=\Big(1-(a+l)\frac{B_l}{A_l}\Big)\Big(1-bp^m(\psi_p(a+l)-\psi_p(1+l))\Big) \mod p^{2m} 
\end{aligned}
$$
where $(*)$ and $(**)$ follow from \cite[Lemma 3.6]{A} and \cite[Lemma 3.8]{A}, respectively.
By arranging the equation above, we obtain
$$
\begin{aligned}
1-(a+l+bp^m)\frac{B_{l+bp^m}}{A_{l+bp^m}}\equiv 1-(a+l)\frac{B_l}{A_l}-bp^m(\psi_p(a+l)-\psi_p(1+l)) &\text{ mod } p^{2m}, \text{when } n>m\\ 
(&\text{mod } p^{n+m}, \text{when } n<m)
\end{aligned}
$$

$$
\begin{aligned}
\implies(a+l+bp^m)(\frac{B_l}{A_l}-\frac{B_{l+bp^m}}{A_{l+bp^m}})\equiv bp^m\bigg(\frac{B_l}{A_l}-(\psi_p(a+l)-\psi_p(1+l))\bigg). 
\end{aligned}
$$

Therefore, by Lemma \ref{BlAl}, we have
$$
\begin{aligned}
(a+l+bp^m)(\frac{B_l}{A_l}-\frac{B_{l+bp^m}}{A_{l+bp^m}})\equiv 0  &\text{ mod } p^{2m}, \text{when } n>m\\ 
(&\text{mod } p^{n+m}, \text{when } n<m).
\end{aligned}
$$

$$
\implies \frac{B_{l+bp^m}}{A_{l+bp^m}}\equiv \frac{B_l}{A_l}\equiv 0 \mod p^m 
$$
in both cases since $\text{ord}_p(a+l+bp^m)$ is $m$ if $m<n$ ($n$ if $n<m$). \medskip

\boxed{\text{Case II : } m=n} \medskip

We write $a+l=cp^m$, $k=l+bp^m$ and $a+k=(b+c)p^m=dp^{m^\prime+m}$ with $p \centernot \mid d$. Here we can assume $m^\prime\geq 1$; otherwise, use the method in Case I. Then
$$
\frac{B_{l+bp^m}}{A_{l+bp^m}}-\frac{B_l}{A_l}=\frac{c[1-(-1)^{e}\frac{A_{bp^{m-1}}^{(1)}}{A_{bp^m+l}}]-dp^{m^\prime}[1-(-1)^{e}\frac{1}{A_l}]}{cdp^{m+m^\prime}} 
.$$
We claim that the numerator
$$
c[1-(-1)^{e}\frac{A_{bp^{m-1}}^{(1)}}{A_{bp^m+l}}]-dp^{m^\prime}[1-(-1)^{e}\frac{1}{A_l}]\equiv 0 \mod p^{2m+m^\prime}.
$$
\boxed{\text{(1) $p$ is odd or $m\geq 2.$}} \medskip

First, we calculate $A_{bp^{m-1}}^{(1)}/{A_{bp^m+l}}$. We have that 
$$
\frac{A_{bp^{m-1}}^{(1)}}{A_{bp^m+l}}=\frac{\{1\}_{l+bp^m}}{\{a\}_{l+bp^m}}=\frac{\{1\}_{l}}{\{a\}_{l}}\frac{\{1+l\}_{bp^m}}{\{a+l\}_{bp^m}}=\frac{1}{A_l}\frac{\{1+l\}_{bp^m}}{\{a+l\}_{bp^m}}
$$
and 
$$
\begin{aligned}
\frac{\{1+l\}_{bp^m}}{\{a+l\}_{bp^m}}=\frac{\prod\limits_{\substack{1+l\leq i \leq bp^m+l \\ p \centernot \mid i}}i}{\prod\limits_{\substack{1\leq i \leq bp^m \\ p\centernot \mid i}}(cp^m+i)}&=\frac{\prod\limits_{\substack{1\leq i \leq bp^m \\ p \centernot \mid i}}i \cdot \prod\limits_{\substack{bp^m\leq i \leq bp^m+l \\ p \centernot \mid i}}i}{\prod\limits_{\substack{1\leq i \leq l \\ p \centernot \mid i}}i}\cdot \frac{1}{\prod\limits_{\substack{1\leq i \leq bp^m \\ p\centernot \mid i}}(cp^m+i)} \\
&=\prod\limits_{\substack{1\leq i \leq bp^m \\ p\centernot \mid i}}\bigg(\frac{1}{1+cp^m/i}\bigg)\frac{\prod\limits_{1\leq i \leq l}(dp^{m+m^\prime}-a-l+i)}{\prod\limits_{1\leq i \leq l}i}.
\end{aligned}
$$

Since $p$ is odd, or $m \geq 2$ and $p\centernot \mid i,$ we have that
$$
\prod\limits_{\substack{1 \leq i \leq bp^m \\ p \centernot \mid i}}\Big(\frac{1}{1+cp^m/i}\Big)=\prod\limits_{\substack{1 \leq i < \frac{bp^m}{2} \\ p \centernot \mid i}}\Big(\frac{1}{1+cp^m/i}\Big)\Big(\frac{1}{1+cp^m/(bp^m-i)}\Big)  
$$ where $bp^m-i\neq i$. Since 

$$
\prod\limits_{\substack{1 \leq i < \frac{bp^m}{2} \\ p \centernot \mid i}}\Big(\frac{1}{1+cp^m/i}\Big)\Big(\frac{1}{1+cp^m/(bp^m-i)}\Big)=\prod_{i}\frac{1}{1+\frac{cp^m\cdot p^m(b+c)}{i(bp^m-i)}}\equiv 1 \mod p^{2m+m^\prime},
$$
we obtain
$$
\begin{aligned}
\frac{\{1+l\}_{bp^m}}{\{a+l\}_{bp^m}}&\equiv \frac{\prod\limits_{1\leq i \leq l}(dp^{m+m^\prime}-a-l+i)}{\prod\limits_{1 \leq i \leq l}i} \mod p^{2m+m^\prime} \\
&=\frac{(-1)^l\prod\limits_{0 \leq i \leq l-1}(a+i-dp^{m+m^\prime})}{l!} \\
&\equiv \frac{(-1)^l}{l!}\Big[\prod\limits_{i=0}^{l-1}(a+i)+(-dp^{m+m^\prime})\sum\limits_{i=0}^{l-1}(a+i)\Big] \\
&\equiv (-1)^l\Big[(-dp^{m+m^\prime})\frac{la+l(l-1)/2}{l!}+\frac{(a)_l}{l!}\Big] \mod p^{2m+m^\prime}.
\end{aligned}
$$

Multiplying $\frac{1}{A_l}$ on both side, we get
$$
\begin{aligned}
\frac{A_{bp^{m-1}}^{(1)}}{A_{bp^m+l}}&\equiv \frac{(-1)^{l}}{A_l}\Big[A_l+(-dp^{m+m^\prime})\frac{la+l(l-1)/2}{l!}\Big] \mod p^{2m+m^\prime} \\
&= (-1)^{l}\Big[1+\frac{l!}{(a)_l}(-dp^{m+m^\prime})\frac{la+l(l-1)/2}{l!}\Big] \\
&\equiv(-1)^{l}\Big[1+(-1)^l(-dp^{m+m^\prime})\frac{la+l(l-1)/2}{l!}\Big] \mod p^{2m+m^\prime}
\end{aligned}
$$ since $\frac{l!}{(a)_l}\equiv (-1)^l \mod p^m$ (use $a\equiv -l \mod p^m$). \medskip

For $1/A_l,$ since
$$
\frac{(a)_l}{l!}\equiv\frac{-cp^m(l+1)l/2+(-1)^ll!}{l!}=\frac{-cp^m(l+1)}{2(l-1)!}+(-1)^l \mod p^{2m},
$$
we have
$$
\frac{1}{A_l}=\frac{l!}{(a)_l}\equiv \frac{(-1)^l}{1+(-1)^{l+1}\frac{cp^m(l+1)}{2(l-1)!}}\equiv (-1)^l+\frac{cp^m(l+1)}{2(l-1)!} \mod p^{2m}.
$$
Hence the numerator
$$
\begin{aligned}
&c[1-(-1)^{e}\frac{A_{bp^{m-1}}^{(1)}}{A_{bp^m+l}}]-dp^{m^\prime}[1-(-1)^{e}\frac{1}{A_l}] \\
=&-b+(-1)^{e}\bigg(\frac{dp^{m^\prime}}{A_l}-c\frac{A_{bp^{m-1}}^{(1)}}{A_{bp^{m}+l}}\bigg) \\
\equiv&-b+(-1)^{e}\bigg((-1)^{l}b+dp^{m^\prime}\frac{cp^m(l+1)}{2[(l-1)!]}+c(dp^{m+m^\prime})\frac{2la+l(l-1)}{2(l!)}\bigg)\\
\equiv&-b+(-1)^{e}\bigg((-1)^{l}b+\frac{dp^{m+m^\prime}c}{2(l!)}(2l(a+l))\bigg) \\
\equiv&b((-1)+(-1)^{(e+l)}) \mod p^{2m+m^\prime}.
\end{aligned}
$$
Since $p$ is odd or $p=2(m\geq 2)$, we have $e=l$. Hence $(-1)+(-1)^{(e+l)}=0$. So we have the numerator congruent to 0 modulo $p^{2m+m^\prime}$. And this implies
$$
\frac{B_k}{A_k}\equiv \frac{B_l}{A_l} \mod p^m.
$$

\boxed{\text{(2) $p=2$ and $m=1.$}} \medskip

Write $k=l+2b, a+l=2c, a+k=2(b+c)=2^{n+1}d$ with $2 \centernot \mid bc$. Then
$$
\begin{aligned}
\frac{A_{b2^{m-1}}^{(1)}}{A_{b2^m}}&=\frac{1}{A_l}\frac{\{1+l\}_{b2^m}}{\{a+l\}_{b2^m}}\\
&=\frac{1}{A_l}\bigg[\prod\limits_{\substack{1\leq i \leq b2^m \\ 2 \centernot \mid i}}\frac{1}{1+c2^m/i}\bigg]\bigg[\frac{(-1)^l(a-d2^{m+n})_l}{l!}\bigg] \\
&\equiv \frac{1}{A_l}(\frac{1}{1+2c/b})(-1)^{l}\bigg(\frac{(a-d2^{m+n})_l}{l!}\bigg) \mod 2^{2+n}.
\end{aligned}
$$
(Note: Only the term $\Big(\frac{1}{1+2c/b}\Big)$ is different from case (1)). \medskip

Again, we have
$$
\begin{aligned}
&(-b)+(-1)^{e}\bigg(\frac{d2^n}{A_l}-c\frac{A_{b}^{(1)}}{A_{2b}}\bigg) \\
\equiv &(-b)+(-1)^{e}\bigg[(-1)^{l}d2^n+d2^n\frac{c2(l+1)}{2(l-1)!}-c(\frac{b}{b+2c})(-1)^{l}-\\
&c(\frac{b}{b+2c})(-d2^{1+n})\frac{2la+l(l-1)}{2(l!)}\bigg] \mod 4\cdot2^n \\
\equiv&(-b)+(-1)^{e}\bigg[(-1)^{l}d2^n-c(\frac{b}{b+2c})(-1)^{l}+d2^n\frac{c2(l+1)}{2(l-1)!}\\
&+c(d2^{1+n})\frac{2la+l(l-1)}{2(l!)}\bigg] \mod 4\cdot2^n \text{ (since } \frac{b}{b+2c}\equiv 1 \mod 2) \\
\equiv&(-b)+(-1)^{e}\bigg[(-1)^{l}d2^n-c(\frac{b}{b+2c})(-1)^{l}\bigg]\\
=&(-b)+(-1)^{(e+l)}\bigg[d2^n-c(\frac{b}{b+2c})\bigg]. \quad (*)
\end{aligned}
$$
Observe that we have $e+l$ is odd. This is because if $l=0,$ then $l^\prime=2$ and $e=1$. If $l=1,$ then $l^\prime=3$ and $e=2$. Hence, we have 
$$
\begin{aligned}
(*)&=(-b)-\Big(b+c-\frac{bc}{b+2c}\Big) \\
&=\frac{(-2)(b+c)^2}{(b+2c)} \\
&\equiv 0 \mod 4\cdot2^n.
\end{aligned}
$$
Therefore, we have
$$
\frac{B_k}{A_k}\equiv\frac{B_l}{A_l} \mod 2.
$$
\end{proof}
Now we can start the proof of Lemma \ref{BA_congruence} with the assumption $\sigma(t)=t^p$ and $s=1$.
\begin{proof}[(proof of Lemma \ref{BA_congruence})]
If $k\not\equiv l \mod p$, then
$$
\frac{B_k}{A_k}=\frac{1}{k+a}.
$$
So the result follows. Suppose $k\equiv l \mod p.$ We write $a+l=cp^n, k=l+bp^m$ with $p \centernot\mid bc$. It is enough to prove the lemma for the case $k^\prime=k+p^{m^\prime}$ for any $m^\prime\in \mathbb{N}$.\medskip

\boxed{\text{Case I: $m^\prime\leq m$}} \medskip

This follows from Lemma \ref{B_k_and_B_l}, since we have $k\equiv k^\prime\equiv l \mod p^{m^\prime}$.\medskip

\boxed{\text{Case II: $m^\prime> m$ and $n\neq m$ or $n=m$ with $\text{ord}_p(k+a)=m$}} \medskip

Since we have 
$$
\begin{aligned}
1-(k^\prime+a)\frac{B_{k^\prime}}{A_k^\prime}&=(-1)^{e}\frac{A_{\lfloor \frac{k^\prime}{p} \rfloor}^{(1)}}{A_{k^\prime}} \\
&=(-1)^{e}\frac{\{1\}_{l+bp^m+p^{m^\prime}}}{\{a\}_{l+bp^m+p^{m^\prime}}} \\
&=(-1)^{e}\frac{\{1\}_{l+bp^m}}{\{a\}_{l+bp^m}} \frac{\{1+l+bp^m\}_{p^{m^\prime}}}{\{a+l+bp^m\}_{p^{m^\prime}}}  \\
&=(-1)^{e}\Big(\frac{A_{bp^{m-1}}^{(1)}}{A_{l+bp^m}}\Big)\cdot \bigg(\frac{\{1+l+bp^m\}_{p^{m^\prime}}}{\{1\}_{p^{m^\prime}}}\frac{\{1\}_{p^{m^\prime}}}{\{a+l+bp^m\}_{p^{m^\prime}}}\bigg)\\
&\equiv\Big(1-(k+a)\frac{B_k}{A_k}\Big)\Big(1+p^{m^\prime}(\psi_p(1+l+bp^m)+\gamma_p)\Big)\cdot \\ &\quad\Big(1-p^{m^\prime}(\psi_p(a+l+bp^m)+\gamma_p)\Big) \mod p^{2m^\prime} \\
&\equiv\Big(1-(k+a)\frac{B_k}{A_k}\Big)\Big(1+p^{m^\prime}(\psi_p(1+l+bp^m)-\psi_p(a+l+bp^m))\Big)\\
&\overset{\mathrm{(*)}}{\equiv}\Big(1-(k+a)\frac{B_k}{A_k}\Big)\Big(1-p^{m^\prime}\frac{B_l}{A_l}\Big) \mod p^{m+m^\prime}
\end{aligned}
$$
where $(*)$ follows from \cite[(2.13)]{A} and Lemma \ref{BlAl} when $p\neq 2$ or $m\geq 2$. And for $p=2, m=1,$ $(*)$ follows from 
$$
\psi_p(1+l+2b)-\psi_p(a+l+2b)\equiv\psi_p(1+l)-\psi_p(a+l) \mod 2.
$$ 
Therefore
$$
\begin{aligned}
1-(k+a+p^{m^\prime})\frac{B_{k^\prime}}{A_{k^\prime}}&\equiv 1-(k+a)\frac{B_k}{A_k}-p^{m^\prime}\frac{B_l}{A_l}+p^{m\prime}(k+a)\frac{B_k}{A_k}\frac{B_l}{A_l} \mod p^{m+m^\prime} \\
&\equiv 1-(k+a)\frac{B_k}{A_k}-p^{m^\prime}\frac{B_l}{A_l} \mod p^{n^{*}+m^\prime} (n^{*}:=min\{n,m\})
\end{aligned}
$$
We get
$$
(k+a)\Big(\frac{B_k}{A_k}-\frac{B_{k^\prime}}{A_{k^\prime}}\Big)\equiv p^{m^\prime}\Big(\frac{B_{k^\prime}}{A_{k^\prime}}-\frac{B_l}{A_l}\Big)\equiv 0 \mod p^{n^{*}+m^\prime}. 
$$
By assumption in Case II, we have that $\text{ord}_p(k+a)=n^{*}$. Hence we obtain our result
$$
\frac{B_k}{A_k}\equiv \frac{B_{k^\prime}}{A_{k^\prime}} \mod p^{m^\prime}.
$$

\boxed{\text{Case III: $m^\prime> m$ and $n=m$ with $\text{ord}_p(k+a)> m$}} \medskip

In this case, we write $a+l=cp^m, k=l+bp^m, k+a=(b+c)p^m=dp^{n+m}$ and $k^\prime=k+p^{m^\prime}$ with $m\geq 1, n\geq 1, p\centernot \mid bcd.$ Then 
$$
\frac{B_{k^\prime}}{A_{k^\prime}}-\frac{B_k}{A_k}=\frac{-p^{m^\prime}+(-1)^{e}\Big[(dp^{m+n}+p^{m^\prime})\frac{A_{bp^{m-1}}^{(1)}}{A_{l+bp^m}}-(dp^{m+n})\frac{A_{bp^{m-1}+p^{m^\prime-1}}^{(1)}}{A_{l+bp^m+p^{m^\prime}}}\Big]}{(dp^{m+n}+p^{m^\prime})dp^{m+n}}
$$
We claim that the numerator congruent to $0$ modulo ${(dp^{m+n}+p^{m^\prime})p^{m+n}}\cdot p^{m^\prime}.$ For simplicity, we denote the numerator by $(*).$ First, we consider \medskip

(1) For $m^\prime\leq m+n,$ we have that
$$
\frac{A_{bp^{m-1}+p^{m^\prime-1}}^{(1)}}{A_{l+bp^m+p^{m^\prime}}}= \frac{\{1\}_{l+bp^m+p^{m^\prime}}}{\{a\}_{l+bp^m+p^{m^\prime}}}
$$ and

$$
\begin{aligned}
\frac{\{1\}_{l+bp^m+p^{m^\prime}}}{\{a\}_{l+bp^m+p^{m^\prime}}}&=\frac{\prod\limits_{\substack{1\leq i \leq l+bp^m+p^{m^\prime} \\ p \centernot \mid i}}i}{\prod\limits_{\substack{1\leq i \leq l+bp^m+p^{m^\prime} \\ p \centernot \mid i}}a+l+bp^m+p^{m^\prime}-i} \\
&=\bigg(\prod\limits_{\substack{1\leq i \leq l+bp^m+p^{m^\prime} \\ p \centernot \mid i}}\frac{1}{1-(dp^{m+n}+p^{m^\prime})/i}\bigg) (-1)^{l+bp^m+p^{m^\prime}-\lfloor\frac{l+bp^m+p^{m^\prime}}{p}\rfloor}\\
&\equiv\bigg(1+(dp^{m+n}+p^{m^\prime})\sum\limits_{\substack{1\leq i \leq l+bp^m+p^{m^\prime} \\ p \centernot \mid i}}\frac{1}{i}\bigg)(-1)^{l+bp^{m-1}(p-1)} \mod (dp^{m+n}+p^{m^\prime})p^{m^\prime} \\
&\quad(\text{since $m^\prime\leq m+n$ and $p^{m^\prime-1}(p-1)\equiv 0 \mod 2$}) \\
&=\Big[1+(dp^{n+m}+p^{m^\prime})(\psi_p(1+l+bp^m+p^{m^\prime})+\gamma_p)\Big](-1)^{l+bp^{m-1}(p-1)}.
\end{aligned}
$$
Therefore
$$
\frac{A_{bp^{m-1}+p^{m^\prime-1}}^{(1)}}{A_{l+bp^m+p^{m^\prime}}}\equiv \Big[1+(dp^{n+m}+p^{m^\prime})(\psi_p(1+l+bp^m+p^{m^\prime})+\gamma_p)\Big](-1)^{l+bp^{m-1}(p-1)}
$$
$\mod (dp^{m+n}+p^{m^\prime})p^{m^\prime}.$ \medskip

Similarly, we can derive
$$
\frac{A_{bp^{m-1}}^{(1)}}{A_{l+bp^m}}\equiv \Big[1+dp^{n+m}(\psi_p(1+l+bp^m)+\gamma_p)\Big](-1)^{l+bp^{m-1}(p-1)} \mod (dp^{m+n})p^{m^\prime}.
$$ \medskip

Hence, $(*)$ congruent to 
$$
-p^{m^\prime}+(-1)^{(e+l+bp^{m-1}(p-1))}\bigg(p^{m^\prime}+\Big(dp^{m+n}+p^{m^\prime}\Big)dp^{m+n}\Big(\psi_p(1+l+bp^m)-\psi_p(1+l+bp^m+p^{m^\prime})\Big)\bigg)
$$
modulo $p^{m+n+m^\prime}(dp^{m+n}+p^{m^\prime})$. \medskip

Furthermore, since $m^\prime\geq2$ we have
$$
\psi_p(1+l+bp^m)\equiv \psi_p(1+l+bp^m+p^{m^\prime}) \mod p^{m^\prime}.
$$
Therefore, $(*)$ congruent to 
$$
p^{m^\prime}\bigg(1-(-1)^{(e+l+bp^{m-1}(p-1))}\bigg).
$$
Now we discuss case by case. \medskip

$\bullet$ If $p$ is odd, then $(-1)^{(e+l+bp^{m-1}(p-1))}=1$. \medskip

$\bullet$ If $p=2, m\geq 2$, then $(-1)^{(e+l+bp^{m-1}(p-1))}=(-1)^{(e+l)}=1$ ($e=l$ since $m\geq 2$).\medskip

$\bullet$ If $p=2, m=1$, then $(-1)^{(e+l+b)}=1$ since it can again be divided into the following two cases:
$$ \left\{
\begin{aligned}
l&=0, l^\prime=2, e=1, b\equiv 1 \mod 2\\
l&=1, l^\prime=3, e=2, b\equiv 1 \mod 2.\\
\end{aligned}
\right.
$$
Hence, $(*)$ congruent to 0 modulo  $p^{m+n+m^\prime}(dp^{m+n}+p^{m^\prime})$, and this implies
$$
\frac{B_k}{A_k}\equiv\frac{B_{k^\prime}}{A_{k^\prime}} \mod p^{m^\prime}.
$$

We proved the first case.\medskip

(2)For $m^\prime>m+n,$ we have that
$$
\begin{aligned}
&\frac{A_{bp^{m-1}+p^{m^\prime-1}}^{(1)}}{A_{l+bp^m+p^{m^\prime}}} \\
=&\frac{\{1\}_{l+bp^m+p^{m^\prime}}}{\{a\}_{l+bp^m+p^{m^\prime}}} \\
=&\frac{A_{bp^{m-1}}^{(1)}}{A_{l+bp^m}}\bigg(\frac{\{1+l+bp^m\}_{p^{m^\prime}}}{\{a+l+bp^m\}_{p^{m^\prime}}} \bigg)\\
=&\frac{A_{bp^{m-1}}^{(1)}}{A_{l+bp^m}}\bigg(\frac{\{1+l+bp^m\}_{p^{m^\prime}}}{\{1\}_{p^{m^\prime}}}\frac{\{1\}_{p^{m^\prime}}}{\{a+l+bp^m\}_{p^{m^\prime}}}\bigg) \\
\equiv&\frac{A_{bp^{m-1}}^{(1)}}{A_{l+bp^m}}\bigg(1+p^{m^\prime}\Big(\psi_p(1+l+bp^m)+\gamma_p\Big)\bigg)\bigg(1-p^{m^\prime}\Big(\psi_p(a+l+bp^m)+\gamma_p\Big)\bigg) \mod p^{2m^\prime}
\end{aligned}
$$
Since $m^\prime>m+n$ and $\psi_p(a+l+bp^m)+\gamma_p\equiv 0 \mod dp^{m+n}$, the formula above congruent to 
$$
\frac{A_{bp^{m-1}}^{(1)}}{A_{l+bp^m}}\bigg(1+p^{m^\prime}\Big(\psi_p(1+l+bp^m)+\gamma_p\Big)\bigg) \mod (dp^{m+n}+p^{m^\prime})p^{m^\prime}.
$$ 

Substitute this result into $(*)$, we have that
$$
(*)\equiv -p^{m^\prime}+(-1)^{e}\frac{A_{bp^{m-1}}^{(1)}}{A_{l+bp^m}}p^{m^\prime}\bigg(1-dp^{m+n}\Big(\psi_p(1+l+bp^m)+\gamma_p\Big)\bigg) \mod (dp^{m+n}+p^{m^\prime})p^{m^\prime}.
$$

Now we calculate the term $A_{bp^{m-1}}^{(1)}/A_{l+bp^m}.$\medskip

Since 
$$
\begin{aligned}
\frac{A_{bp^{m-1}}^{(1)}}{A_{l+bp^m}}&=\frac{\{1\}_{l+bp^m}}{\{a\}_{l+bp^m}} \\
&= \prod\limits_{\substack{1\leq i \leq l+bp^m \\ p \centernot \mid i}}\bigg(\frac{i}{a+l+bp^m-i}\bigg) \\
&=\bigg(\prod\limits_{\substack{1\leq i \leq l+bp^m \\ p \centernot \mid i}}\frac{1}{1-dp^{m+n}/i}\bigg)(-1)^{l+bp^{m-1}(p-1)} \\
&=\bigg(\prod\limits_{\substack{1\leq i \leq l+bp^m \\ p \centernot \mid i}}\frac{1}{1-dp^{m+n}/i}\bigg)(-1)^{e} \text{\quad \Big($(-1)^{l+bp^{m-1}(p-1)}=(-1)^e$ $\forall p$\Big)  }\\
&\equiv (-1)^{e}\bigg(1+dp^{m+n}\Big(\psi_p(1+l+bp^m)+\gamma_p\Big)\bigg) \mod (p^{m+n})^2,
\end{aligned}
$$

we obtain
$$
\begin{aligned}
(*)&\equiv -p^{m^\prime}+p^{m^\prime}\bigg(1+dp^{m+n}\Big(\psi_p(1+l+bp^m)+\gamma_p\Big)\bigg) \bigg(1-dp^{m+n}\Big(\psi_p(1+l+bp^m)+\gamma_p\Big)\bigg) \\
&\equiv 0 \mod (dp^{m+n}+p^{m^\prime})p^{m+n}p^{m^\prime}.
\end{aligned}
$$
Hence, again, we obtain 
$$
\frac{B_k}{A_k}\equiv\frac{B_{k^\prime}}{A_{k^\prime}} \mod p^{m^\prime}.
$$
Now we prove that $B_k/A_k\in W.$ For $k\not\equiv l\mod p$, $B_k/A_k=1/(k+a)\in W.$
For $k\equiv l \mod p$, it follows from $A_l \in \mathbb{Z}_p^{\times}$ and 
$$
\frac{B_k}{A_k}\equiv \frac{B_l}{A_l} \mod p.
$$ 
\end{proof}
\subsection{Proof of Congruence Relations: End of the Proof}
Here we follow the method used in the paper \cite[\S 3.5]{A} where M. Asakura prove the congruence relations of $p$-adic hypergeometric functions of logarithmic type. \medskip

Here we prove it in a more general setting.
\begin{lem}\label{generalBA}
Let $a\in\mathbb{Z}_p\backslash \mathbb{Z}_{\leq 0}$ and let $B_n\in \mathbb{Z}_p$ for $n\in \mathbb{Z}_{\geq 0}.$ Suppose the map $n\mapsto B_n/A_n$ is a Lipschitz function. Put $B(t):=\sum_{n=0}^\infty B_n t^n$. Then there are Dwork's congruences for $B(t)/F_{a,\cdots,a}(t);$ that is
$$
\frac{B(t)}{F_{a,\cdots,a}(t)}\equiv \frac{B(t)_{<p^n}}{F_{a,\cdots,a}(t)_{<p^n}} \mod p^n W[[t]]
$$ for all $n\in \mathbb{Z}_{\geq 0}$.
\end{lem}

If we can prove this lemma, then by Lemma \ref{BA_congruence} we obtain the congruence relations for our hypergeometric functions. \medskip

Put $S_m:=\sum_{i+j=m}A_{i+p^n}B_j-A_iB_{j+p^n}$ for $m\in \mathbb{Z}_{\geq 0}.$ We claim that $S_m\equiv 0 \mod p^n$ for $n\in\mathbb{Z}_{\geq}$ since Lemma \ref{generalBA} is equivalent to for any $n, m\in \mathbb{Z}_{\geq 0}$, $S_m\equiv 0 \mod p^n.$ First, we need the following lemmas.
\begin{lem}\label{S_m_1}
$$
S_m\equiv \sum_{i+j=m}(A_{i+p^n}A_j-A_iA_{j+p^n})\frac{B_j}{A_j} \mod p^n.
$$
\end{lem}
\begin{proof}
\begin{equation*}
\begin{aligned}
S_m&=\sum_{i+j=m}A_{i+p^n}B_j-A_iA_{j+p^n}\frac{B_{j+p^n}}{A_{j+p^n}} \\
&\equiv \sum_{i+j=m}A_{i+p^n}B_j-A_iA_{j+p^n}\frac{B_{j}}{A_{j}} \mod p^n \quad \text{(Lipschitz)}\\
&=\sum_{i+j=m}(A_{i+p^n}A_j-A_iA_{j+p^n})\frac{B_j}{A_j} .
\end{aligned}
\end{equation*}

\end{proof}

\begin{lem}
$$
S_m\equiv \sum_{i+j=m}(A_{\lfloor j/p \rfloor}^{(1)}A_{\lfloor i/p \rfloor+p^{n-1}}^{(1)}-A_{\lfloor i/p \rfloor}^{(1)}A_{\lfloor j/p \rfloor+p^{n-1}}^{(1)})\frac{A_i}{A_{\lfloor i/p \rfloor}^{(1)}}\frac{A_j}{A_{\lfloor j/p \rfloor}^{(1)}}\frac{B_j}{A_j} \mod p^n.
$$
\end{lem}
\begin{proof}
By Lemma \ref{S_m_1}, we know that
$$
\begin{aligned}
S_m&\equiv\sum_{i+j=m}(A_{i+p^n}A_j-A_iA_{j+p^n})\frac{B_j}{A_j} \mod p^n \\
&=\sum_{i+j=m}\bigg(\frac{A_{i+p^n}A^{(1)}_{\lfloor i/p\rfloor}A^{(1)}_{\lfloor j/p\rfloor}}{A_i}-\frac{A_{j+p^n}A^{(1)}_{\lfloor i/p\rfloor}A^{(1)}_{\lfloor j/p\rfloor}}{A_j}\bigg)\frac{A_i}{A^{(1)}_{\lfloor i/p\rfloor}}\frac{A_j}{A^{(1)}_{\lfloor j/p\rfloor}}\frac{B_j}{A_j}. \\
\end{aligned}
$$
By \cite[Lemma 3.8]{A}, we have
$$
\frac{A_{i+p^n}A^{(1)}_{\lfloor i/p \rfloor}}{A_i}\equiv A^{(1)}_{\lfloor i/p \rfloor +p^{n-1}}, \quad \frac{A_{j+p^n}A^{(1)}_{\lfloor j/p \rfloor}}{A_j}\equiv A^{(1)}_{\lfloor j/p \rfloor +p^{n-1}} \mod p^n.
$$
Since $A_i/A^{(1)}_{\lfloor i/p\rfloor}, A_j/A^{(1)}_{\lfloor j/p\rfloor}$ and $B_j/A_j$ are all in $W$ (\cite[Lemma 3.8]{A} and assumption), we have $S_m$ is congruent to
$$
\sum_{i+j=m}(A_{\lfloor j/p \rfloor}^{(1)}A_{\lfloor i/p \rfloor+p^{n-1}}^{(1)}-A_{\lfloor i/p \rfloor}^{(1)}A_{\lfloor j/p \rfloor+p^{n-1}}^{(1)})\frac{A_i}{A_{\lfloor i/p \rfloor}^{(1)}}\frac{A_j}{A_{\lfloor j/p \rfloor}^{(1)}}\frac{B_j}{A_j} \mod p^n.
$$
\end{proof}
Now we start the proof of Lemma \ref{generalBA}.

\begin{proof}[(proof of Lemma \ref{generalBA})] \medskip

Put
$$
q_i:=\frac{A_i}{A_{\lfloor i/p \rfloor}^{(1)}}, \quad A(i,j):=A_i^{(1)}A_j^{(1)}, \quad A^{*}(i,j):=A(j,i+p^{n-1})-A(i,j+p^{n-1}) 
$$
$$
B(i,j):=A^{*}(\lfloor i/p \rfloor , \lfloor j/p \rfloor).
$$

Put $m=pt+l$ with $l\in \{0,1,\cdots, p-1\}.$
Then
$$
\begin{aligned}
S_m &\equiv \sum_{i+j=m}B(i,j)q_iq_j\frac{B_j}{A_j} \mod p^n \\
&=\sum\limits_{i=0}^{p-1}\sum\limits_{k=0}^{\lfloor (m-i)/p \rfloor}B(i+kp, m-(i+kp))q_{i+kp}q_{m-(i+kp)}\frac{B_{m-(i+kp)}}{A_{m-(i+kp)}} \\
&= \sum\limits_{k=0}^{t}\sum\limits_{i=0}^{l}B(i+kp, m-(i+kp))q_{i+kp}q_{m-(i+kp)}\frac{B_{m-(i+kp)}}{A_{m-(i+kp)}} \\
&\quad + \sum\limits_{k=0}^{t-1}\sum\limits_{i=l+1}^{p-1}B(i+kp, m-(i+kp))q_{i+kp}q_{m-(i+kp)}\frac{B_{m-(i+kp)}}{A_{m-(i+kp)}} \\
&=\sum\limits_{k=0}^{t}A^{*}(k,t-k)\overbrace{\bigg(\sum\limits_{i=0}^{l}q_{i+kp}q_{m-(i+kp)}\frac{B_{m-(i+kp)}}{A_{m-(i+kp)}}}^{P_k}\bigg) \\
&\quad + \sum\limits_{k=0}^{t-1}A^{*}(k,t-k-1)\underbrace{\bigg(\sum\limits_{i=l+1}^{p-1}q_{i+kp}q_{m-(i+kp)}\frac{B_{m-(i+kp)}}{A_{m-(i+kp)}}}_{Q_k}\bigg).
\end{aligned}
$$

We show that the first term vanishes modulo $p^n$. \medskip

It follows from assumption and \cite[Lemma 3.7]{A} that $B_k/A_k, q_k\in W$ for all $k\in\mathbb{Z}_{\geq 0}$ and
$$
k\equiv k^\prime \mod p^i \quad \Longrightarrow \frac{B_k}{A_k}\equiv \frac{B_{k^\prime}}{A_{k^\prime}}, q_k\equiv q_{k^\prime} \mod p^i.
$$
 
Therefore, we have
\begin{equation}
k\equiv k^\prime \mod p^i \quad \Longrightarrow \quad P_k \equiv P_{k^\prime} \mod p^{i+1}.
\end{equation}

Then one can write
$$
\sum\limits_{k=0}^{s}A^{*}(k,s-k)P_k\equiv \sum\limits_{i=0}^{p^{n-1}-1}P_i\overbrace{\bigg(\sum\limits_{k\equiv i \text{ mod } p^{n-1}}A^{*}(k,s-k)\bigg)}^{(*)} \mod p^n.
$$

Let us recall Lemma 3.12 in \cite{A}.
\begin{lem}
For all $m,k,s \in \mathbb{Z}_{\geq 0}$ and $0\leq l \leq n$, then
$$
\sum\limits_{\substack{i+j=m \\ i\equiv k \text{ mod } p^{n-l}}}A_iA_{j+p^{n-1}}-A_jA_{i+p^{n-1}}\equiv 0 \mod p^l.
$$
\end{lem}

Using this lemma, we obtain $(*)$ is $0$ modulo $p$. Hence, by (3.1) again, we can write
$$
\sum\limits_{k=0}^{s}A^{*}(k,s-k)P_k\equiv \sum\limits_{i=0}^{p^{n-2}-1}P_i\overbrace{\bigg(\sum\limits_{k\equiv i \text{ mod } p^{n-2}}A^{*}(k,s-k)\bigg)}^{(**)} \mod p^n.
$$

It follows from Lemma 3.9 that $(**)$ is $0$ modulo $p^2$, so
$$
\sum\limits_{k=0}^{s}A^{*}(k,s-k)P_k\equiv \sum\limits_{i=0}^{p^{n-3}-1}P_i\bigg(\sum\limits_{k\equiv i \text{ mod } p^{n-3}}A^{*}(k,s-k)\bigg) \mod p^n.
$$
Continuing the same discussion, we have
$$
\sum\limits_{k=0}^{s}A^{*}(k,s-k)P_k\equiv \sum\limits_{k=0}^{s}A^{*}(k,s-k)P_0 =0 \mod p^n.
$$
Similarly, we can show the vanishing of the second term
$$
\sum\limits_{k=0}^{s-1}A^{*}(k,s-k)Q_k\equiv 0 \mod p^n.
$$
Hence $S_m \equiv 0 \mod p^n.$ And this is the end of the proof.
\end{proof}

\section{Transformation Formulas}
In this section, we will introduce two conjectures and prove these conjectures in a particular case. The first one is ``Transformation Formulas between $p$-adic hypergeometric functions $\widehat{\mathscr{F}}_{a,...,a}^{(\sigma)}(t)$ and $p$-adic hypergeometric functions of logarithmic type''
which will be discussed in Section 4.3. The second one is ``Transformation formula on Dwork's $p$-adic Hypergeometric Functions'' which is discussed in Section 4.4.  
\subsection{Hypergeometric Curves and Hypergeometric Curves of Gauss Type}
The main reference of this section is \cite[\S 4.1, \S 4.2, \S 4.6]{A} \medskip
\subsubsection{Hypergeometric Curves}\label{Hypergeometric Curves}
Let $W=W(\overline{\mathbb{F}}_p)$ be the Witt ring of $\overline{\mathbb{F}}_p$ and $R=W[t,(t-t^2)^{-1}].$ Let $N\geq 2$ be an integer and a prime $p>N.$ We denote by $\mathbb{P}^1_R(Z_0,Z_1)$ the projective line over R with homogeneous coordinate $(Z_0,Z_1).$ Then we 
define $U$ to be 
$$
{\rm Spec}R[u,v]/((1-u^N)(1-v^N)-t)\subset Y:=\overline{U}
$$
where $Y$ is the closure in $\mathbb{P}^1_R\times \mathbb{P}^1_R.$
 It is called a \textit{hypergeometric curves} over $R.$ The morphism $Y \rightarrow {\rm Spec}R$ is smooth projective with connected fibers of relative dimension one and the genus of a geometric fiber is $(N-1)^2.$ \medskip
 
\begin{lem}
There is a morphism
$$
\overline{f}: \overline{Y} \rightarrow \mathbb{P}^1_W=\mathbb{P}^1_W(T_0,T_1)
$$ of smooth projective $W$-schemes satisfying \medskip

$(1)$ Let $S:={\rm Spec}R={\rm Spec}W[t,(t-t^2)^{-1}] \subset \mathbb{P}^1_W$ with $t:=T_1/T_0.$ Then $Y=\overline{f}^{-1}(S)\rightarrow S$ is the hypergeometric curve. \medskip

$(2)$ $\overline{f}$ has a semistable reduction at $t=0.$ The fiber $D:=\overline{f}^{-1}(t=0)$ is a relative simple NCD, and the multiplicities of the components are one.
\end{lem}
\begin{proof}
This is \cite[Lemma 4.1]{A}.
\end{proof}

Let $K={\rm Frac}W(\overline{\mathbb{F}}_p).$ For a $W$-scheme $Z$ and a $W$-algebra $R,$ we write $Z_R=Z\times_W R.$ The group $\mu_N(K)\times \mu_N(K)$ acts on $Y$ in the following way 
$$
[\zeta, \nu]\cdot (x,y,t)=(\zeta x, \nu y, t), \quad (\zeta, \nu)\in \mu_N\times\mu_N.
$$ One has the eigen decomposition
$$
H^1_{\rm dR}(Y_K/S_K)=\bigoplus_{i=1}^{N-1}\bigoplus_{j=1}^{N-1}H^1_{\rm dR}(Y_K/S_K)(i,j)
$$ where $H^1_{\rm dR}(Y_K/S_K)(i,j)$ denote the subspace on which $(\zeta, \nu)$ acts by multiplication by $\zeta^i\nu^j$ for all $(\zeta, \nu).$ Then each eigenspace $H^1_{\rm dR}(Y_K/S_K)(i,j)$ is free of rank $2$ over $\mathcal{O}(S_K)$. Put  
\begin{equation}
a_i:=1-\frac{i}{N}, \quad b_j:=1-\frac{j}{N}, 
\end{equation}
\begin{equation}
   \omega_{i,j}:=N\frac{x^{i-1}j^{j-N}}{1-x^N}dx= -N\frac{x^{i-N}y^{j-1}}{1-y^N}dy,
\end{equation} and
\begin{equation}
   \eta_{i,j}:=\frac{1}{x^N-1+t}\omega_{i,j}=Nt^{-1}x^{i-N}y^{j-N-1}dy
\end{equation} for integers $i,j$ such that $1\leq i, j\leq N-1.$ Then they forms a $\mathcal{O}(S_K)$-free basis of $H^1_{\rm dR}(Y_K/S_K)(i,j)$ (\cite[Lemma 2.2]{A2}). \medskip

Let 
$$
F_{a,b}(t)=\sum_{i=0}^\infty \frac{(a)_i}{i!}\frac{(b)_i}{i!}t^i \in K[[t]]
$$ be the hypergeometric series. Put
\begin{equation}
    \Tilde{\omega}_{i,j}:=\frac{1}{F_{a_i,b_j}(t)}\omega_{i,j}, \quad \Tilde{\eta}_{i,j}:=-t(1-t)^{a_i+b_j}(F^\prime_{a_i,b_j}(t)\omega_{i,j}+b_jF_{a_i,b_j}(t)\eta_{i,j}).
\end{equation}

\begin{lem}\label{ker_nabla}
Let $\nabla_{i,j}$ be the connection on the eigen component $H_{i,j}:=K((t))\otimes_{\mathcal{O}_S}H^1_{\rm dR}(Y_K/S_K)(i,j).$ Then ${\rm Ker} \nabla_{i,j}=K\widetilde{\eta}_{i,j}.$ Moreover let $\overline{\nabla}_{i,j}$ be the connection on $H_{i,j}/K((t))\widetilde{\eta}_{i,j}$ induced from $\nabla_{i,j}.$ Then ${\rm Ker} \overline{\nabla}_{i,j}=K\widetilde{\omega}_{i,j}.$
\end{lem}
\begin{proof}
See \cite[Proposition 4.3, Corollary 4.4]{A}.
\end{proof}

Let $(i,j)=(q/r,q^\prime/ r^\prime)\in \mathbb{Q}^2$ such that $\gcd(r,N)=\gcd(r^\prime, N)=1$ and $q, q^\prime$ are not divided by $N.$ Then we define
\begin{equation}
 H^1_{\rm dR}(Y_K/S_K)(i,j)= H^1_{\rm dR}(Y_K/S_K)(i_0,j_0), \quad \omega_{i,j}=\omega_{i_0,j_0}, \cdots, \widetilde{\eta}_{i,j}=\widetilde{\eta}_{i_0,j_0}  
\end{equation} where $i_0,j_0$ are the unique integers such that $i_0\equiv i \mod N$, $j_0\equiv j \mod N$ and $1\leq i_0, j_0<N.$

\subsubsection{Hypergeometric Curves of Gauss Type}\label{Hypergeometric Curves of Gauss Type}
Let $A, B$ be integers such that $1\leq A , B < N$ and gcd$(A,N)$=gcd$(B,N)$=1. \medskip

Let $f:Y\rightarrow \mathbb{P}^1$ be a fibration over $K$ $(K={\rm Frac}W(\overline{\mathbb{F}}_p))$ whose general fiber $Y_{t}=f^{-1}(t)$ is the projective nonsingular curve associated to the affine curve
$$
y^N=x^A(1-x)^B(1-(1-t)x)^{N-B}
$$
We call $X$ a hypergeometric curve of Gauss type. This is a fibration of curves of genus $N-1,$ smooth outside $t=0,1,\infty.$ Put $S_K:={\rm Spec}K[t,(t-t^2)^{-1}]$ and $X_0:=f^{-1}(S_K).$\medskip

Let $[\zeta]:X_0\rightarrow X_0$ be the automorphism given by
$$
[\zeta](x,y,t)=(x,\zeta^{-1}y,t)
$$ for any $N$-th root of unity $\zeta\in \mu_{N}=\mu_{N}(K).$ We denote
$$
H_{{\rm dR}}^{1}(X_0/S_K)(n):=\{x\in H_{{\rm dR}}^{1}(X_0/S_K) \ | \ [\zeta]x=\zeta^n x, \forall \zeta \in \mu_N\}.
$$
Then one has the eigen decomposition 
$$
H_{{\rm dR}}^{1}(X_0/S_K)=\bigoplus_{n=1}^{N-1}H_{{\rm dR}}^{1}(X_0/S_K)(n)
$$ of $\mathscr{O}(S_K)$-module and each eigen space is free of rank $2$. A basis of $H_{{\rm dR}}^{1}(X_0/S_K)(n)$ (\cite[Lemma 2.5]{A2}) is given by
$$
\omega_n:=x^{A_n}(1-x)^{B_n}(1-(1-t)x)^{n-1-B_n}\frac{dx}{y^n}, \quad \eta_n:=\frac{x}{1-(1-t)x}\omega_n
$$
where
$$
A_n:=\lfloor\frac{nA}{N} \rfloor, \quad B_n:=\lfloor\frac{nB}{N} \rfloor.
$$ \medskip

Let 
$$
a_n:=\bigg\{\frac{-nA}{N}\bigg\}, b_n:=\bigg\{\frac{-nB}{N}\bigg\}
$$ where $\{x\}:=x-\lfloor x \rfloor.$
Put
\begin{equation}
    \widetilde{\omega}_n:=\frac{1}{F_{a_n,b_n}(t)}\omega_n, \quad \widetilde{\eta}_n:=-t(1-t)^{a_n+b_n}(F^\prime_{a_n,b_n}(t)\omega_n+b_jF_{a_n,b_n}(t)\eta_n).
\end{equation} which form a $K((t))$-basis of $K((t))\otimes H^1_{\rm dR}(X_0/S_K).$\medskip

Let $Y/S$ be the hypergeometric curves in \S 4.1.1. Then we have the following Lemma.

\begin{lem}{$(\cite[\S 4.6]{A})$}\label{quotient_map} Let $g_{A,B}$ be the automorphism of $Y_K/S_K$ given by $g_{A,B}: (u,v)\mapsto (\zeta_N^Bu,\zeta_N^{-A}v)$ for a fixed primitive $N$-root of unity in $\mu_N(K)$. Let $G=\langle g_{A,B}\rangle\subset {\rm Aut}(Y_K/S_K)$ be the cyclic group of order $N.$ Then the quotient $Y_G:=Y_K/G$ by $G$ is isomorphic to the curve $X: y^N=x^A(1-x)^B(1-(1-t)x)^{N-B}.$ The quotient map $\rho: Y_K \rightarrow X$ is given by  
$$
(u,v)\longmapsto (x,y)=(u^{-N}, u^{-A}(1-u^{-N})v^{N-B}) 
$$ which is a $S$-morphism (i.e. $t\mapsto t$).
\end{lem}
Put
$$
\widehat{\omega}_n=\sum_{g \in G}g^*\omega_{nA,nB} \quad \widehat{\eta}_n=\sum_{g \in G}g^*\eta_{nA,nB}.
$$

\begin{lem}\label{omega}
Using the notation above, we have
\begin{equation}
    \omega_{nA,nB}=\rho^*(\omega_n), \quad  \eta_{nA,nB}=\rho^*(\eta_n)
\end{equation}
\begin{equation}
    \widetilde{\omega}_{nA,nB}=\rho^*(\widetilde{\omega}_n), \quad \widetilde{\eta}_{nA,nB}=\rho^*(\widetilde{\eta}_n)
\end{equation}
\begin{equation}
   \widehat{\omega}_n=N\rho^*(\omega_n), \quad  \widehat{\eta}_n=N\rho^*(\eta_n).
\end{equation}
\end{lem}
Then we see that the pull-back $\rho^*$ satisfies
\begin{equation}\label{pullback}
   \rho^*(H^1_{\rm dR}(X_0/S_K)(n))=H^1_{\rm dR}(Y_K/S_K)(nA,nB), \quad 0<n<N
\end{equation}
and the push-forward $\rho_*$ satisfies
\begin{equation}\label{pushforward}
H^1_{\rm dR}(Y_K/S_K)(i,j) =\begin{cases}
H^1_{\rm dR}(X_0/S_K)(n) \quad (i,j)\equiv (nA,nB) \mod N \\
0 \quad \text{otherwise}
\end{cases} 
\end{equation}
for $0<i,j<N.$
One has 
\begin{equation}
   \rho_*(\omega_{nA,nB}), \quad \rho_*(\eta_{nA,nB})
\end{equation} a basis of $H^1_{\rm dR}(X_0/S_K)(n)$ and 
$$
\rho^*\rho_*(\omega_{nA,nB})=N\omega_{nA,nB}.
$$ from (\ref{pullback}), (\ref{pushforward}) and $\rho_*\rho^*=N.$
\subsection{Review of paper \cite{A}}
In this part, we will give a brief review of what we need in the paper \cite{A}.
First, let us recall the definition of $p$-adic hypergeometric functions of logarithmic type.
\begin{defi}[\textbf{$p$-adic hypergeometric functions of logarithmic type}]\label{p-adic hypergeometric functions of logarithmic type}
Let $s\geq 1$ be a positive integer. Let $(a_1,\cdots, a_s)\in \mathbb{Z}_p$ and $(a_1^\prime,\cdots,a_s^\prime)$ where $a_{i}^\prime$ denotes the Dwork prime. Let $\sigma : W[[t]]\rightarrow W[[t]]$ be the $p$-th Frobenius endomorphism given by $\sigma(t)=ct^p$ with $c\in 1+pW.$ Then we define the $p$-adic hypergeometric functions of logarithmic type to be
$$
\mathscr{F}^{(\sigma)}_{a_1,\cdots,a_s}(t):=\frac{1}{F_{a_1,\cdots,a_s}(t)}\bigg[\psi_p(a_1)+\cdots+\psi_p(a_s)+s\gamma_p-p^{-1}\log(c)+\int_{0}^{t}(F_{a_1,\cdots,a_s}(t)-F_{a_1^\prime,\cdots,a_s^\prime}(t^{\sigma}))\frac{dt}{t}\bigg]
$$ where $\psi_p(z)$ is the $p$-adic digamma function( \cite[\S 2.2]{A}), 
$\gamma_p$ is the $p$-adic Euler constant (\cite[\S 2.2]{A})
and $\log(z)$ is the Iwasawa logarithmic function.
\end{defi}

\begin{defi}[cf. {\cite[3.1.1]{LP}}]
Let $W=W(\overline{k})$ where $\overline{k}$ is an algebraic closed field with ${\rm char}\overline{k}>0.$
Define
$$
\begin{aligned}
W\langle t_1,\cdots, t_n\rangle&:=\varprojlim_{n}W/p^n[t_1,\cdots,t_n]\\
&=\bigg\{\sum_{I}a_It^I\in W[[t_1,\cdots,t_n]] \ \Big| \ |a_I|_p\rightarrow 0 \ as \ |I|\rightarrow \infty \bigg\}
\end{aligned}
$$
$$
W[t_1,\cdots,t_n]^{\dagger}:=\bigg\{\sum_{I}a_It^I\in W[[t_1,\cdots,t_n]] \ \Big| \ \exists r>1 \ such \ that \ |a_I|_pr^{|I|}\rightarrow 0 \ as \ |I|\rightarrow \infty \bigg\}.
$$
Furthermore, if 
$$
A=W[t_1,\cdots,t_n]/(f_1,\cdots,f_r),
$$
we define the weak completion of $A$ to be
\[
A^{\dagger}:=W[t_1,\cdots,t_n]^{\dagger}/(f_1,\cdots,f_r).
\]
\end{defi} 

\begin{defi}
Let $F:W\rightarrow W$ be the $p$-th Frobenius. Then $\sigma :A^{\dagger}\rightarrow A^{\dagger}$ is called a $p$-th Frobenius if 
\begin{itemize}
    \item $\sigma$ is $F$-linear $(\sigma(\alpha x)=F(\alpha)\sigma(x), \alpha\in W, x\in A^{\dagger})$
    \item $\sigma \mod p$ on $A^{\dagger}/pA^{\dagger}\simeq A/pA$ is given by $x\rightarrow x^p.$
\end{itemize}
\end{defi}
Let $\sigma$ be a $p$-th Frobenius on $W[t,(t-t^2)^{-1}]^{\dagger}$ which extends on $K[t,(t-t^2)^{-1}]^{\dagger}:=K\otimes W[t,(t-t^2)^{-1}]^{\dagger}.$ Write $X_{\overline{\mathbb{F}}_p}:=X_W\times_{W} \overline{\mathbb{F}}_p$ and $S_{\overline{\mathbb{F}}_p}:=S_W\times_{W} \overline{\mathbb{F}}_p.$ Then the \textit{rigid cohomology} groups
$$
H_{{\rm rig}}^{\bullet}(X_{\overline{\mathbb{F}}_p}/S_{\overline{\mathbb{F}}_p})
$$
are defined. We refer the book \cite{L} for the theory of rigid cohomology.\medskip

The required properties is the following.
\begin{itemize}
    \item $H_{{\rm rig}}^{\bullet}(X_{\overline{\mathbb{F}}_p}/S_{\overline{\mathbb{F}}_p})$ is a finitely generated $\mathscr{O}(S_K)^{\dagger}$-module.
    \item (Frobenius) The $p$-th Frobenius $\Phi$ on $H_{{\rm rig}}^{\bullet}(X_{\overline{\mathbb{F}}_p}/S_{\overline{\mathbb{F}}_p})$(depending on $\sigma$) is defined. This is a $\sigma$-linear endomorphism:
    $$
    \Phi(f(t)x)=\sigma(f(t))\Phi (x), {\rm for \ } x\in H_{{\rm rig}}^{\bullet}(X_{\overline{\mathbb{F}}_p}/S_{\overline{\mathbb{F}}_p}), f(t)\in \mathscr{O}(S_K)^{\dagger}.
    $$
    \item (Comparison) There is the comparison isomorphism with algebraic de Rham cohomology,
    $$
    c: H_{{\rm rig}}^{\bullet}(X_{\overline{\mathbb{F}}_p}/S_{\overline{\mathbb{F}}_p})\cong H_{{\rm dR}}^{\bullet}(X_0/S_K)\otimes_{\mathscr{O}(S_K)}\mathscr{O}(S_K)^{\dagger}.
    $$
\end{itemize}

\begin{thm}\label{unique_lifting}
Let $\sigma$ be a $p$-Frobenius on $\mathscr{O}(S_K)^{\dagger}$ such that $\sigma(t)=ct^p$ with $c\in 1+pW.$ Then there exists an exact sequence
$$
0\longrightarrow \mathscr{O}(S_K)^{\dagger}\otimes_{\mathscr{O}(S_K)}H_{\rm dR}^1(X_0/S_K)\longrightarrow \mathscr{O}(S_K)^{\dagger}\otimes_{\mathscr{O}(S_K)}M_{\xi}(X_0/S_K)\longrightarrow \mathscr{O}(S_K)^{\dagger} \longrightarrow 0
$$
endowed with 
\begin{itemize}
\item Frobenius $\Phi$-action which is $\sigma$-linear.
\item ${\rm Fil}^i$ $\subseteq$ $M_{\xi}(X_0/S_K)$ $({\rm Hodge \ filtration})$ with
$$
\mathscr{O}(S_K)^{\dagger}\otimes_{\mathscr{O}(S_K)}{\rm Fil}^0M_{\xi}(X_0/S_K) \xrightarrow{\sim} \mathscr{O}(S_K)^{\dagger}
$$
In particular, there exists a unique lifting  $e_{\xi}\in \mathscr{O}(S_K)^{\dagger}\otimes_{\mathscr{O}(S_K)}{\rm Fil}^0M_{\xi}(X_0/S_K)$ of $1\in \mathscr{O}(S_K)^{\dagger}.$
\end{itemize}
\end{thm}
\begin{proof}
Recall $U={\rm Spec}R[u,v]/((1-u^N)(1-v^N)-t)\subset Y:=\overline{U}$
where $Y$ is the closure in $\mathbb{P}^1_R\times \mathbb{P}^1_R.$
For $(\nu_1,\nu_2)\in \mu_N (W)\times \mu_N(W)$, put 
$$
\xi(\nu_1,\nu_2)=\bigg\{\frac{u-1}{u-\nu_2}, \frac{v-1}{v-\nu_2}\bigg\}\in K_2^M (\mathcal{O}(U)).
$$

Then according to \cite[\S 2.6]{AM2}, we have the $1$-extension
\begin{equation}\label{1-ext_of_Y}
 0\longrightarrow H^1(Y/S)(2) \longrightarrow M_{\xi(\nu_1,\nu_2)}(Y/S)\longrightarrow \mathcal{O}_S\longrightarrow 0 
\end{equation} in the category of Fil-$F$-MIC$(S,\sigma).$\medskip

Let $\xi:=\sum_{i=0}^{N-1}(g_{A,B}^i)^*\xi(\nu_1,\nu_2)\in K_2^M (\mathcal{O}(U))$ which is fixed under the action of $G,$ so that $G$ acts on $M_\xi(Y/S)$. Taking the fixed part of (\ref{1-ext_of_Y}) by $\langle g_{A,B} \rangle,$ we have a $1$-extension
\[
\begin{CD}
0 @>>>  H^1(Y/S)^G(2) @>>>  M_\xi(Y/S)^G @>>> {\mathcal{O}}_S @>>> 0\\
  @. @| @| @. @. \\
  @. H @. M_\xi 
\end{CD}
\] 
By Lemma \ref{quotient_map}, it follows
\begin{itemize}
    \item $H_{\rm dR}\simeq H_{\rm dR}^1(X_0/S_K)$
    \item $H_{\rm rig}\simeq \mathcal{O}(S_K)^\dagger \otimes H_{\rm dR}^1(X_0/S_K)$ 
    \item ${\rm Fil}^i H_{\rm dR}={\rm Fil}^{i+2} H_{\rm dR}^1(X_0/S_K)$ where $\rm Fil^\bullet$ in the right denotes the Hodge filtration. In particular, ${\rm Fil}^0 H_{\rm dR}=0.$
\end{itemize}
Therefore there is the unique element $e_\xi\in {\rm Fil}^0M_\xi$ which is a lifting of $1\in \mathcal{O}(S_K).$
\end{proof}

Let $1\leq n \leq N-1$ be an integer and $A, B$ be integers such that $1\leq A , B < N$ and gcd$(A,N)$=gcd$(B,N)$=1. \medskip

Put
$$
a_n:=\bigg\{\frac{-nA}{N}\bigg\}, \quad b_n:=\bigg\{\frac{-nB}{N}\bigg\}
$$
where $\{x\}:=x-\lfloor x\rfloor$ denotes the fractional part. Let
$$
F_n(t):=\sum\limits_{i=0}^{\infty}\frac{(a_n)_i}{i!}\frac{(b_n)_i}{i!}t^i\in\mathbb{Z}_p[[t]]
$$
be the hypergeometric power series. Put
$$
e^{\rm unit}_{i,j}:=(1-t)^{-a_i-b_j}F_{a_i,b_j}(t)^{-1}\widetilde{\eta}_{i,j}.
$$ 
Using the group $G$ in Lemma \ref{quotient_map}, we define
$$
\widehat{\omega}_n=\sum_{g\in G}g^*\omega_{nA,nB}, \quad e_n^{\rm unit}:=\sum_{g\in G}g^* e_{nA,nB}^{\rm unit}.
$$
Let $e_{\xi}$ be the unique lifting of $1$ in the above theorem. 
\begin{lem}\label{dlog(xi)}
$$
\nabla (e_\xi)=-\sum_{n=1}^{N-1}\frac{(1-\nu_1^{-nA})(1-\nu_2^{-nB})}{N^2}\frac{dt}{t}\widehat{\omega}_n.
$$
\end{lem}
\begin{proof}
Let $\xi:=\sum_{i=0}^{N-1}g_{A,B}^i\xi(\nu_1,\nu_2)$ with
\begin{equation}\label{K2symbol}
 \xi(\nu_1,\nu_2)=\bigg\{\frac{x-1}{x-\nu_1}, \frac{y-1}{y-\nu_2}\bigg\}. 
\end{equation}
Following from the compatibility of the connection with respect to the pull-back, we have
$$
\nabla(e_{\xi})=\sum_{i=0}^{N-1}(g_{A,B}^*)^i\nabla(e_{\xi(\nu_1,\nu_2)}).
$$ Then by \cite[\S 4.4 (4.25)]{A}, we have 
$$
\nabla(e_{\xi(\nu_1,\nu_2)})=-d\log(\xi(\nu_1,\nu_2))=-\sum_{i=1}^{N-1}\sum_{j=1}^{N-1}\frac{(1-\nu_1^{-i})(1-\nu_2^{-j})}{N^2}\frac{dt}{t}\omega_{i,j}.
$$
So we have
$$
\begin{aligned}
\nabla(e_{\xi})&=\sum_{i=0}^{N-1}(g_{A,B}^*)^i\nabla(e_{\xi(\nu_1,\nu_2)}) \\
&= -\sum_{i=1}^{N-1}\sum_{j=1}^{N-1}\frac{(1-\nu_1^{-i})(1-\nu_2^{-j})}{N^2}\frac{dt}{t} \sum_{g\in  G}g^*\omega_{i,j}\\
&= -\sum_{n=1}^{N-1}\frac{(1-\nu_1^{-nA})(1-\nu_2^{-nB})}{N^2}\frac{dt}{t}\sum_{g\in  G}g^* \omega_{nA,nB} \\
&= -\sum_{n=1}^{N-1}\frac{(1-\nu_1^{-nA})(1-\nu_2^{-nB})}{N^2}\frac{dt}{t}\widehat{\omega}_n.
\end{aligned}
$$
\end{proof}

\begin{lem}\label{e_xi_1}
Assume $\sigma$ is given by $\sigma(t)=ct^p$ with $c\in 1+pW.$ Let $h(t)=\prod_{m=0}^s F_{a_n^{(m)}, b_n^{(m)}}(t)_{<p}$ where $s$ is the minimal integer such that $(a_{n}^{(s+1)},b_{n}^{(s+1)})=(a_n,b_n)$ for all $n\in \{1,2,\cdots,N-1\}.$ Then 
$$
e_\xi-\Phi(e_\xi)\equiv -\sum_{n=1}^{N-1}\frac{(1-\nu_1^{-nA})(1-\nu_2^{-nB})}{N^2}\mathscr{F}_{a_n,b_n}^{(\sigma)}(t)\widehat{\omega}_n 
$$ modulo $\sum_{n=1}^{N-1}K\langle t,(t-t^2)^{-1},h(t)^{-1}\rangle e_n^{\rm unit}.$
\end{lem}
\begin{proof}
See \cite[Theorem 4.19]{A}.
\end{proof}
\subsection{Transformation Formula}
Let $a_i\in \mathbb{Z}_p(0\leq i \leq r-1)$ and put $h(t):=\prod_{i=0}^{r-1}F_{a_i,\cdots,a_i}(t)_{<p},$ where $F_{a_i,\cdots,a_i}(t)$ is hypergeometric power series.
Then there is an involution
$$
\omega : W\langle t,t^{-1},h(t)^{-1}\rangle \longrightarrow W\langle t,t^{-1},h(t)^{-1}\rangle, \quad \omega(f(t))=f(t^{-1}).
$$
This follows from the following proposition.
\begin{prop}\label{involution}
$(1)$ Let $a\in \mathbb{Z}_p.$ Let $F(t):=F_{a,\cdots,a}(t)_{<p} \mod p.$ Then $F(t)=(-1)^{ls} t^lF(t^{-1})$ in $\mathbb{F}_p[t]$ where $l$ is the degree of $F(t)$ which equals the unique integer in $\{0,1,\cdots,p-1\}$ such that $a+l\equiv 0 \mod p.$ \medskip

$(2)$ Let $a_i\in \mathbb{Z}_p(0\leq i \leq r-1)$ and put $h(t):=\prod_{i=0}^{r-1}F_{a_i,\cdots,a_i}(t)_{<p}.$ Then there is a ring homomorphism 
$$
\omega_n: W/p^n[t,t^{-1},h(t)^{-1}]\rightarrow W/p^n[t,t^{-1},h(t)^{-1}], \quad f(t)\mapsto f(t^{-1}).
$$\medskip
$(3)$ There is an involution
$$
\omega : W\langle t,t^{-1},h(t)^{-1}\rangle \longrightarrow W\langle t,t^{-1},h(t)^{-1}\rangle, \quad \omega(f(t))=f(t^{-1}).
$$ 
\end{prop}
\begin{proof}
(1)Write 
$$
F(t)=\sum\limits_{i=0}^{l}\bigg(\frac{(a)_i}{i!}\bigg)^st^i.
$$ Since $F(t)=F_{a,\cdots,a}(t)_{<p} \mod p,$ we have that$(a)_l\not\equiv 0 \mod p$ and $(a)_{l+1}\equiv 0 \mod p.$ That is, we have $a+l\equiv 0 \mod p.$\medskip

If $l=i+j$, then
$$
\begin{aligned}
\frac{(a)_i}{i!}&\equiv \frac{(-l)_i}{i!}=(-1)^i\binom{l}{i}=(-1)^i\binom{l}{j}=(-1)^l\frac{(-l)_j}{j!}\equiv (-1)^l\frac{(a)_j}{j!} \mod p.
\end{aligned}
$$
Therefore, we have 
$$
\bigg(\frac{(a)_i}{i!}\bigg)^s\equiv (-1)^{ls} \bigg(\frac{(a)_j}{j!}\bigg)^s \mod p.
$$
This implies 
$$
t^lF(t^{-1})=(-1)^{ls}F(t).
$$
(2) Observe that it is enough to show that 
$$
h(t^{-1})\in (W/p^n[t,t^{-1},h(t)^{-1}])^{\times}.
$$
An element is a unit in $W\langle t,t^{-1},h(t)^{-1}\rangle$ if and only if it is a unit modulo $pW[ t,t^{-1},h(t)^{-1}].$ So if we can prove $h(t^{-1})$ is a unit modulo $pW[ t,t^{-1},h(t)^{-1}],$
then we are done. 
Put
$$
F_i(t):=F_{a_i,\cdots,a_i}(t)_{<p} \mod p.
$$
From (1), we have
$$
F_i(t)=\pm t^{l_i}F_i(t^{-1})
$$ where $l_i$ is the degree of $F_i(t).$
Hence we have 
$$
F_i(t^{-1})^{-1}=\frac{\pm t^{l_i}}{F_i(t)} 
$$ in $W/p[t,t^{-1},h(t)^{-1}],$ i.e. $F_i(t)$ is a unit in $W/p[t,t^{-1},h(t)^{-1}].$ Since $h(t)$ is the product of $F_i(t),$ it is also a unit.\medskip

(3) $\omega$ is defined by using $\omega_n$ in (2) in the following way:
$$
\omega : W\langle t,t^{-1},h(t)^{-1}\rangle \longrightarrow W\langle t,t^{-1},h(t)^{-1}\rangle, \quad (f_n(t))\mapsto (\omega_nf_n(t)).
$$
By (2), $\omega$ is well-defined and also it is an involution.
\end{proof}
\begin{conj}[\textbf{Transformation Formula between $\mathscr{F}_{a,\cdots,a}^{\;(\sigma)}(t)$ and $\widehat{\mathscr{F}}_{a,\cdots,a}^{\;(\widehat{\sigma})}(t)$}] 
\label{transformation-conj}
Let $\sigma(t)=ct^p$ and $\widehat{\sigma}(t)=c^{-1}t^p$. Let $a\in\mathbb{Z}_p\backslash \mathbb{Z}_{\leq 0}$ and the $r$th Dwork prime $a^{(r)}=a$ for some $r>0.$ Put $h(t):=\prod_{i=0}^{r-1}F_{a^{(i)},\cdots,a^{(i)}}(t)_{<p},$ then
$$
\mathscr{F}_{a,\cdots,a}^{\;(\sigma)}(t)=-\widehat{\mathscr{F}}_{a,\cdots,a}^{\;(\widehat{\sigma})}(t^{-1})
$$
in the ring $W\langle t,t^{-1},h(t)^{-1}\rangle$ 
where $\widehat{\mathscr{F}}_{a,\cdots,a}^{\;(\widehat{\sigma})}(t^{-1})$ is defined as $\omega(\widehat{\mathscr{F}}_{a,\cdots,a}^{\;(\widehat{\sigma})}(t))$ and $\mathscr{F}_{a,\cdots,a}^{\;(\sigma)}(t)$ is the $p$-adic hypergeometric functions of logarithmic type.
\end{conj}
Before moving to the next theorem, we prove, as an example, that this conjecture is true modulo $p.$ \medskip
\begin{eg}
By Theorem \ref{congruence-thm} and \cite[Theorem 3.3]{A}, we know that 
$$
\mathscr{F}_{a,\cdots,a}^{\;(\sigma)}(t)\equiv \frac{G^{(\sigma)}_{a,\cdots,a}(t)_{<p}}{F_{a,\cdots,a}(t)_{<p}}, \quad \widehat{\mathscr{F}}_{a,\cdots,a}^{\;(\widehat{\sigma})}(t^{-1})\equiv \left.\frac{\widehat{G}_{a,\cdots,a}^{(\widehat{\sigma})}(t)_{<p}}{F_{a,\cdots,a}(t)_{<p}}\right|_{t^{-1}}\mod pW[[t]].
$$
Let $F(t)=F_{a,\cdots,a}(t)_{<p},$ $G(t)\equiv G^{(\sigma)}_{a,\cdots,a}(t)_{<p}$ and $\widehat{G}(t)\equiv\widehat{G}_{a,\cdots,a}^{(\widehat{\sigma})}(t)_{<p} \mod p.$ 

Then by Proposition \ref{involution}, one has 
$$
\frac{G(t)}{F(t)}+\frac{\widehat{G}(t^{-1})}{F(t^{-1})}=\frac{G(t)}{(-1)^{ls}t^lF(t^{-1})}+\frac{\widehat{G}(t^{-1})}{F(t^{-1})}=\frac{G(t)+(-1)^{ls}t^l\widehat{G}(t^{-1})}{(-1)^{ls}t^lF(t^{-1})}.
$$ We claim that $G(t)+(-1)^{ls}t^l\widehat{G}(t^{-1})\equiv 0 \mod p.$ Indeed, since 
$$
B_k=A_k\frac{B_k}{A_k}, \quad \widehat{B}_k=A_k\frac{\widehat{B}_k}{A_k} \equiv 0 \mod p
$$when $l+1 \leq k \leq p-1$, we have $G(t)=\sum_{i=0}^l B_k t^k$ and $\widehat{G}(t)=\sum_{i=0}^l \widehat{B}_kt^k.$ Thus 
$$
G(t)+(-1)^{ls}t^l\widehat{G}(t^{-1})=\sum_{i=0}^l B_k t^k+(-1)^{ls}\sum_{i=0}^l \widehat{B}_{l-k}t^k=\sum_{i=0}^l \bigg(B_k+(-1)^{ls}\widehat{B}_{l-k}\bigg)t^k.
$$ Then for $1\leq k\leq l$, we have
$$
B_k+(-1)^{ls}\widehat{B}_{l-k}=\frac{A_k}{k}+(-1)^{ls}\frac{A_{l-k}}{a+l-k}\equiv \frac{1}{k}\bigg(A_k-(-1)^{ls}A_{l-k}\bigg)\equiv 0 \mod p.
$$ When $k=0,$ 
$$
\begin{aligned}
&B_0+(-1)^{ls}\widehat{B}_l \\
=&s(\psi_p(a)+\gamma_p)-p^{-1}\log(c)-(-1)^{ls}\frac{1}{a+l}\bigg(A_l-c^{\frac{l+a}{p}}(-1)^{se}\bigg) \\
=&s(\psi_p(a)+\gamma_p)-(-1)^{ls}\frac{A_l}{a+l}-p^{-1}\log(c)+(-1)^{s(e+l)}\frac{c^{\frac{l+a}{p}}}{a+l}.
\end{aligned}
$$ Write $c=1+pz$ and $a+l=dp^n$ with $p\nmid d.$ Then 
$$
-p^{-1}\log(c)+(-1)^{s(e+l)}\frac{c^{\frac{l+a}{p}}}{a+l}\equiv (-1)^{s(e+l)}\frac{1}{a+l} \mod p.
$$ So 
$$
\begin{aligned}
&B_0+(-1)^{ls}\widehat{B}_l \\
\equiv &s(\psi_p(a)+\gamma_p)+(-1)^{ls}\frac{A_l-(-1)^{se}}{a+l}\\
\overset{(*)}{\equiv} &s(\psi_p(-l)+\gamma_p)-s(\psi_p(1+l)+\gamma_p) \mod p
\end{aligned}
$$ where $(*)$ follows from \cite[(2.13)]{A} and imitate the proof of Lemma \ref{BlAl}.
Since $\psi_p(-l)=\psi_p(1+l)$ (cf. \cite[Theorem 2.4 (2)]{A}), we obtain
$$
B_0+(-1)^{ls}\widehat{B}_l\equiv 0 \mod p.
$$
\end{eg}
~\medskip
Recall $a_n=\{-nA/N\}$ in Section 4.2. Let $r\in\mathbb{Z}_{\geq 1}$ is a number such that $a_n^{(r)}=a_n.$ We will prove a special case of Conjecture 4.9 in this section.

\begin{thm}\label{transformation-thm}
Let $\sigma(t)=ct^p$ and $\widehat{\sigma}(t)=c^{-1}t^p$. Put $h(t)=\prod_{i=0}^{r-1}F_{a_n^{(i)},a_n^{(i)}}(t)_{<p}.$ Then
$$
\mathscr{F}_{a_n,a_n}^{\;(\sigma)}(t)=-\widehat{\mathscr{F}}_{a_n,a_n}^{\;(\widehat{\sigma})}(t^{-1})
$$
in the ring $W\langle t,t^{-1},h(t)^{-1}\rangle$ 
where $\mathscr{F}_{a_n,a_n}^{\;(\sigma)}(t)$ is the $p$-adic hypergeometric functions of logarithmic type.
\end{thm}
Before giving the proof, we need some settings. First, we assume $A=B$ and let $t_0^N=t.$ Then we have two descriptions of the quotient curve $Y_G/S_K$ (Lemma \ref{quotient_map}): \medskip

(i) $Y_G\simeq X$ with coordinates $(x,y,t_0),$

$$
X: y^N=x^A(1-x)^A(1-(1-t_0^N)x)^{N-A}
$$
(ii) $Y_G\simeq \widehat{X}$ with coordinates $(z,w,s_0),$ 
$$
\widehat{X}: w^N=z^A(1-z)^A(1-(1-s_0^N)z)^{N-A}
$$ where $z=1-x, w=t_0^{A-N}y, s_0=t_0^{-1}.$

\begin{lem}\label{e_xi_2}
Assume $\widehat{\sigma}$ is given by $\widehat{\sigma}(t)=c^{-1}t^p$ with $c\in 1+pW.$ Let $h(t)=\prod_{m=0}^s F_{a_n^{(m)}, b_n^{(m)}}(t)_{<p}$ where $s$ is the minimal integer such that $(a_{n}^{(s+1)},b_{n}^{(s+1)})=(a_n,b_n)$ for all $n\in \{1,2,\cdots,N-1\}.$ Then 
$$
e_\xi-\Phi(e_\xi)\equiv \sum_{n=1}^{N-1}\frac{(1-\nu_1^{-nA})(1-\nu_2^{-nB})}{N^2}\widehat{\mathscr{F}}_{a_n,b_n}^{(\widehat{\sigma})}(s_0^N)\widehat{\omega}_n 
$$ modulo $\sum_{n=1}^{N-1}K\langle t,(t-t^2)^{-1},h(t)^{-1}\rangle e_n^{\rm unit}.$
\end{lem}
\begin{proof}
We consider the quotient map (Lemma \ref{quotient_map}) $Y_G\simeq \widehat{X}$ in coordinate $(z,w,s_0).$ From the coordinate $(z,w,s_0),$ let $e_\xi^{\widehat{X}}$ be the unique lifting in Theorem \ref{unique_lifting}. We  write $\widehat{e}_\xi$ for $e_\xi^{\widehat{X}}$ via the change of coordinate $(z,w,s_0)=(1-x, t_0^{A-N}y, t_0^{-1}),$ and $\widehat{\Phi}$ is the Frobenius action which is $\widehat{\sigma}$-linear.  Write
\begin{equation}\label{hat(e)}
 \widehat{e}_\xi-\widehat{\Phi}(\widehat{{e}}_\xi)=\sum_{n=1}^{N-1}\frac{(1-\nu_1^{-nA})(1-\nu_2^{-nB})}{N}\widehat{E}_1^{(n)}(t_0){\widetilde{\omega}}_n+\widehat{E}_2^{(n)}(t_0)\widetilde{\eta}_n\in K((t))\otimes H^1_{\rm dR}({X}/S_K).  
\end{equation}
Apply the Gauss-Manin connection $\nabla$ on (\ref{hat(e)}). Since $\nabla \widehat{\Phi}=\widehat{\Phi}\nabla$ and  
$$
\nabla (\widehat{e}_\xi)=-\sum_{n=1}^{N-1}\frac{(1-\nu_1^{-nA})(1-\nu_2^{-nA})}{N^2}\frac{dt}{t}t^{a_n}(N\omega_n),
$$ from Lemma \ref{dlog(xi)} and change of coordinate $(z,w,s_0)=(1-x, t_0^{A-N}y, t_0^{-1}),$
we have

\begin{equation}\label{equation_2}
\begin{aligned}
&-\sum\limits_{n=1}^{N-1}\frac{(1-\nu_1^{-nA})(1-\nu_2^{-nA})}{N}(1-\widehat{\Phi})\bigg(t^{a_n}F_n(t)\frac{dt}{t}\wedge{\widetilde{\omega}}_n\bigg) \\
=&\sum\limits_{n=1}^{N-1}\frac{(1-\nu_1^{-nA})(1-\nu_2^{-nA})}{N}\nabla(\widehat{E}_1^{(n)}(t_0){\widetilde{\omega}}_n+\widehat{E}_2^{(n)}(t_0)\widetilde{\eta}_n).
\end{aligned}
\end{equation}\medskip

By \cite[Proposition 4.7]{A}, we have $\widehat{\Phi}({\widetilde{\omega}}_m)\equiv p^{-1}{\widetilde{\omega}}_n \mod K((t_0))\widetilde{\eta}_{n}$ where $m$ is the unique integer in $\{1,2,...,N-1\}$ such that $pm\equiv n \mod N$.\medskip

Therefore
$$
\text{LHS of (\ref{equation_2})}\equiv -\sum\limits_{n=1}^{N-1}\frac{(1-\nu_1^{-nA})(1-\nu_2^{-nA})}{N}\bigg[t^{a_n}F_n(t)-\big(t^{a_m}F_m(t)\big)^{\widehat{\sigma}}\bigg]\frac{dt}{t}\wedge\widetilde{{\omega}}_n. 
$$

On the other hand, it follows from \cite[Proposition 4.3]{A}, we have
$$
\text{RHS of (\ref{equation_2})}\equiv \sum\limits_{n=1}^{N-1}\frac{(1-\nu_1^{-nA})(1-\nu_2^{-nA})}{N}\frac{d\widehat{E}_1^{(n)}}{dt}dt\wedge \widetilde{\omega}_n \mod K((t_0))\widetilde{\eta}_n.
$$
Thus 
$$
\frac{d\widehat{E}_1^{(n)}}{dt}=-\frac{t^{a_n}F_n(t)-\big(t^{a_m}F_m(t)\big)^{\widehat{\sigma}}}{t}.
$$
Namely, we have
$$
-\widehat{E}_1^{(n)}(t_0)=C+\int_{0}^{t}t^{a_n}F_n(t)-\big(t^{a_m}F_m(t)\big)^{\widehat{\sigma}}\frac{dt}{t}
$$
for some constant $C\in K$ and with $t=t_0^N.$ We claim that this constant $C$ is $0$. Indeed, since
$\widehat{E}_1^{(n)}(t_0)/F_n(t_0^N)$ is an overconvergent function, 
$\widehat{E}_1^{(n)}(t_0)t_0^{-Na_n}/F_n(t_0^N)$ is also overconvergent. If $C=0$, then 
$\widehat{E}_1^{(n)}(t_0)t_0^{-Na_n}/F_n(t_0^N)=\widehat{\mathscr{F}}_{a_n,a_n}^{\;(\widehat{\sigma})}(t)$ is a 
convergent function by Corollary \ref{overconvergent}. If there is another $C^\prime$ such that $\widehat{E}_1^{(n)}(t_0)t_0^{-Na_n}/F_n(t_0^N)$ is a convergent function, then after subtraction, we have
$$
\frac{C^\prime t_0^{-Na_n}}{F_n(t_0^N)}\in K\langle t_0, (t_0-t_0^2)^{-1}, h(t_0^N)^{-1} \rangle.
$$ This is a contradiction (as it is shown in the proof of \cite[proof in Theorem 4.9]{A}). So $C$ must be $0$. Therefore we have
$$
\widehat{e}_\xi-\widehat{\Phi}(\widehat{{e}}_\xi)=\sum_{n=1}^{N-1}\frac{(1-\nu_1^{-nA})(1-\nu_2^{-nB})}{N}[-\widehat{\mathscr{F}}_{a_n,a_n}^{(\widehat{\sigma})}(t)\cdot t^{a_n}]{{\omega}}_n \mod \widetilde{\eta}_n.
$$ 
Now we write everything back to coordinate $(z,w,s_0).$
Since ${\rm Ker}\nabla$ is generated by $\{\widetilde{\eta}_n\}$ over $K$ (Lemma \ref{ker_nabla}) and ${\rm Ker}\nabla$ contains in ${\rm Ker}\nabla$ via change of coordinate, we have
$$
{e}_\xi^{\widehat{X}}-{\Phi}({{e}}_\xi^{\widehat{X}})=\sum_{n=1}^{N-1}\frac{(1-\nu_1^{-nA})(1-\nu_2^{-nB})}{N^2}[\widehat{\mathscr{F}}_{a_n,a_n}^{(\widehat{\sigma})}(s_0^{-N})]N{{\omega}}_n \mod \widetilde{\eta}_n.
$$ in coordinate $(z,w,s_0).$ Therefore from Lemma \ref{omega}, we have
$$
e_\xi-\Phi(e_\xi)\equiv \sum_{n=1}^{N-1}\frac{(1-\nu_1^{-nA})(1-\nu_2^{-nB})}{N^2}\widehat{\mathscr{F}}_{a_n,b_n}^{(\widehat{\sigma})}(t^{-1})\widehat{\omega}_n 
$$ modulo $\sum_{n=1}^{N-1}K\langle t,(t-t^2)^{-1},h(t)^{-1}\rangle e_n^{\rm unit}.$
This completes the proof.
\end{proof}

Using these lemmas, we can prove Theorem \ref{transformation-thm}.
\begin{proof}[(proof of Theorem \ref{transformation-thm})]

By Lemma \ref{e_xi_1}, we know
\begin{equation}\label{equation_e_xi_1}
e_{\xi}-\Phi(e_{\xi})\equiv\sum\limits_{n=1}^{N-1}\frac{(1-\nu_1^{-nA})(1-\nu_2^{-nA})}{N^2}\big[-\mathscr{F}_{a_n,a_n}^{(\sigma)}(t)\big]\widehat{\omega}_n 
\end{equation} modulo $\sum_{n=1}^{N-1} K\langle t,(t-t^2)^{-1}, h(t)^{-1}\rangle e_n^{\rm unit}$.\medskip

Also, by Lemma \ref{e_xi_2}, we have
\begin{equation}\label{equation_e_xi_2}
e_{\xi}-\Phi(e_{\xi})\equiv \sum\limits_{n=1}^{N-1}\frac{(1-\nu_1^{-nA})(1-\nu_2^{-nA})}{N^2}\big[\widehat{\mathscr{F}}_{a_n,a_n}^{(\widehat{\sigma})}(t^{-1})\big]\widehat{\omega}_n 
\end{equation} modulo $\sum_{n=1}^{N-1} K\langle t,(t-t^2)^{-1}, h(t)^{-1}\rangle e_n^{\rm unit}$.
\medskip

Comparing equations (\ref{equation_e_xi_1}) and (\ref{equation_e_xi_2}), we have
$$
-\widehat{\mathscr{F}}_{a_n,a_n}^{(\widehat{\sigma})}(t)=\mathscr{F}_{a_n,a_n}^{(\sigma)}(t^{-1}).
$$
This completes the proof.
\end{proof}
\subsection{Transformation formula on Dwork's $p$-adic Hypergeometric Functions}
In this section, the settings are the same as in Section 4.3. First, let us recall the definition of Dwork's $p$-adic hypergeometric functions.
\begin{defi}[\textbf{Dwork's $p$-adic hypergeometric functions}]\label{Dwork's $p$-adic hypergeometric functions}
Let $s \geq 1$ be an integer. For $(a_1,\cdots,a_s)\in \mathbb{Z}_p^s$, Dwork's $p$-adic hypergeometric function is defined as
$$
 \mathscr{F}^{{\rm Dw}}_{a_1,\cdots,a_s}(t):=F_{a_1,\cdots,a_s}(t)/F_{a_1^\prime,\cdots,a_s^\prime}(t^p)
$$
where $F_{a_1,\cdots,a_s}(t)$ and $F_{a_1^\prime,\cdots,a_s^\prime}(t)$ are hypergeometric power series.
\end{defi}
\begin{remark}
Dwork's $p$-adic hypergeometric function belongs to $W\langle t, f(t)^{-1}\rangle$ where 
$$
f(t):=\prod_{i=0}^N F_{a_1^{(i)},\cdots, a_s^{(i)}}(t)_{<p}
$$ and $N$ is an integer such that
$$
\{\overline{F_{a_1^{(i)},\cdots, a_s^{(i)}}(t)_{<p}} \; | \; i\in \mathbb{Z}_{\geq 0}\}=\{\overline{F_{a_1^{(i)},\cdots, a_s^{(i)}}(t)_{<p}} \; | \; i=0,1,\cdots,N\}
$$ with $\overline{f(t)}=f(t) \mod p$ (consequence of Dwork's congruence). 
\end{remark}

\begin{conj}[\textbf{Transformation formula on Dwork's $p$-adic Hypergeometric Functions}]\label{Dwork-trans-conj}
Suppose $a_1=\cdots=a_s=a.$ Let $l$ is the unique integer in $\{0,1,\cdots,p-1\}$ such that $a+l\equiv 0 \mod p.$ Then
$$
\mathscr{F}^{{\rm Dw}}_{a,\cdots,a}(t)=((-1)^st)^{l}\mathscr{F}^{{\rm Dw}}_{a,\cdots,a}(t^{-1}).
$$
\end{conj}

\begin{remark}
In Proposition \ref{involution}, we have Conjecture \ref{Dwork-trans-conj} holds modulo $p.$
\end{remark}

We have the special case of Conjecture \ref{Dwork-trans-conj}. 
\begin{thm}\label{Dwork-trans-thm}
Let $a\in \frac{1}{N}{\mathbb{Z}}, 0<a<1$ and $p>N$. We have
$$
\mathscr{F}^{{\rm Dw}}_{a,a}(t)=t^{l}\mathscr{F}^{{\rm Dw}}_{a,a}(t^{-1}),
$$
where $l$ is the unique integer in $\{0,1,\cdots,p-1\}$ such that $a+l\equiv 0 \mod p.$
\end{thm}
\begin{proof}
Let $\sigma:W[[t]]\rightarrow W[[t]]$ be the $p$-th Frobenius given by $\sigma(t)=t^p,$ and let $\Phi$ be the $p$-th Frobenius induced by $\sigma$ on $K((t_0))\otimes H^1_{{\rm dR}}(Y/S)$ where $Y/S$ is the hypergeometric curve in \S \ref{Hypergeometric Curves}. By \cite[Proposition 4.7]{A}, $\Phi$ induces a map
with 
$$
\Phi(\widetilde{\omega}_{p^{-1}nA,p^{-1}nA})\equiv p^{-1}\widetilde{\omega}_{nA,nA} \mod K((t))\widetilde{\eta}_{nA,nA}.
$$ 
Let $m \equiv p^{-1}nA \mod N$ with $1\leq m \leq N-1$ and $k\equiv nA \mod N$ with $1\leq n \leq N-1.$ Then we have
\begin{equation}
\Phi({\omega}_{p^{-1}nA,p^{-1}nA})\equiv p^{-1}\frac{F_m(t^{\sigma})}{F_k(t)}{\omega}_{nA,nA}=p^{-1}\mathscr{F}^{{\rm Dw}}_{a_n,a_n}(t)^{-1}{\omega}_{nA,nA}
\end{equation} where $a_n=\{-nA/N\}.$ \medskip

Therefore
\begin{equation}\label{midequation}
    \Phi(\widehat{\omega}_m)\equiv p^{-1}\mathscr{F}^{{\rm Dw}}_{a_n,a_n}(t)^{-1}\widehat{\omega}_n.
\end{equation}
Recall that we have two description of $Y/G.$ From the coordinate $(z,w,s_0),$ one has 
\begin{equation}
    \Phi(\widetilde{\omega}_m)\equiv p^{-1}\mathscr{F}^{{\rm Dw}}_{a_n,a_n}(s_0^{N})^{-1}\widetilde{\omega}_n.
\end{equation}
Then using $(z,w,s_0)=(1-x, t_0^{A-N}y,t_0^{-1}),$ we obtain
\begin{equation}\label{compare_1}
    \Phi(-t^{a_m}\widetilde{\omega}_m)\equiv p^{-1}\mathscr{F}^{{\rm Dw}}_{a_n,a_n}(t^{-1})^{-1}(-t^{a_n})\widetilde{\omega}_n.
\end{equation} 
On the other hand, we have 
\begin{equation}\label{compare_2}
\Phi(-t^{a_m}\widetilde{\omega}_m)\equiv (-t^{pa_m})p^{-1}\mathscr{F}^{{\rm Dw}}_{a_n,a_n}(t)^{-1}\widetilde{\omega}_n
\end{equation} by $\Phi(t)=t^p$ and (\ref{midequation}). Then comparing (\ref{compare_1}) and (\ref{compare_2}), we obtain the result.
\end{proof}

\end{document}